\documentclass[11pt]{article}

\usepackage{geometry}
\usepackage{latexsym}
\usepackage{mathrsfs}
\usepackage{amsfonts}
\usepackage{amssymb}
\usepackage{subfigure}    %for parallel figures
\usepackage{graphicx}
\usepackage{epstopdf}
\usepackage{float}
\usepackage{multirow}
\usepackage{bm}
\usepackage{amsmath}
\usepackage{amsthm}
\usepackage{color}
\usepackage[subnum]{cases}
\usepackage{rotating}
\usepackage{enumitem}
\usepackage{accents}
\usepackage{algorithm}
\usepackage{algorithmic}
\usepackage[title]{appendix}
\usepackage{mathtools}
\usepackage{caption}
\usepackage{url}
\usepackage{booktabs}
\usepackage{pifont}

\makeatletter
\@addtoreset{equation}{section}
\makeatother

\makeatletter
\def \wideubar{\underaccent{{\cc@style\underline{\mskip15mu}}}}
\def \widebar{\accentset{{\cc@style\underline{\mskip10mu}}}}
\makeatother

\definecolor{blue}{rgb}{0,0,0.9}
\definecolor{red}{rgb}{0.9,0,0}
\definecolor{green}{rgb}{0,0.9,0}
\definecolor{brown}{rgb}{0.6,0.1,0.1}
\definecolor{lightgreen}{rgb}{0.1,0.5,0.1}

\usepackage[colorlinks=true,
breaklinks=true,
bookmarks=true,
urlcolor=blue,
citecolor=lightgreen,
linkcolor=lightgreen,
bookmarksopen=false,
draft=false]{hyperref}

\graphicspath{{images/}}

\begin{document}

\newtheorem{property}{Property}[section]
\newtheorem{proposition}{Proposition}[section]
\newtheorem{append}{Appendix}[section]
\newtheorem{definition}{Definition}[section]
\newtheorem{lemma}{Lemma}[section]
\newtheorem{corollary}{Corollary}[section]
\newtheorem{theorem}{Theorem}[section]
\newtheorem{remark}{Remark}[section]
\newtheorem{problem}{Problem}[section]
\newtheorem{example}{Example}[section]
\newtheorem{assumption}{Assumption}
\renewcommand*{\theassumption}{\Alph{assumption}}

\title{Inexact Accelerated Proximal Gradient Method Revisit: 
\\ An Economical Variant via Shadow Points}

\author{Lei Yang\thanks{School of Computer Science and Engineering, and Guangdong Province Key Laboratory of Computational Science, Sun Yat-sen University ({\tt yanglei39@mail.sysu.edu.cn}). },\qquad
Meixia Lin\thanks{(Corresponding author) Engineering Systems and Design, Singapore University of Technology and Design ({\tt meixia\_lin@sutd.edu.sg}). }
}

%\date{Started from May 19, 2022}
%\date{}
\maketitle

\begin{abstract}
We consider the problem of optimizing the sum of a smooth convex function and a non-smooth convex function via the inexact accelerated proximal gradient (APG) method. A key limitation of existing inexact APG methods is their reliance on \textit{feasible} approximate solutions of the subproblems, which is often computationally expensive or even unrealistic to obtain in practice. To overcome this limitation, we develop a shadow-point enhanced inexact APG method (SpinAPG), which relaxes the feasibility requirement by allowing the computed iterates to be potentially infeasible, while introducing an auxiliary feasible \textit{shadow point} solely for error control without requiring its explicit computation. This design decouples feasibility enforcement from the algorithmic updates and leads to a flexible and practically implementable inexact framework. Under suitable summable error conditions, we show that SpinAPG preserves all desirable convergence properties of the APG method, including the iterate convergence and an $o(1/k^2)$ convergence rate for the objective function values. These results complement and extend existing convergence analyses of inexact APG methods by demonstrating that the accelerated convergence can be retained even in the presence of controlled infeasibility. Numerical experiments on sparse quadratic programming problems illustrate the practical advantages of SpinAPG, showing that it can substantially reduce computational overhead by avoiding explicit computations of feasible points.

\vspace{5mm}

%\noindent {\bf AMS subject classifications.} 90C05, 90C06, 90C25

\end{abstract}

%%%%%%%%%%%%%%%%%%%%%%%%%%%%%%%%%%%%%%%%%%%%
\section{Introduction}

In this paper, we consider the following convex composite optimization problem:
\begin{equation}\label{mainpro}
\min\limits_{\bm{x}\in\mathbb{E}}~~F(\bm{x}) := P(\bm{x}) + f(\bm{x}),
\end{equation}
where $\mathbb{E}$ is a real finite-dimensional Euclidean space endowed with an inner product $\langle\cdot,\cdot\rangle$ and its induced norm $\|\cdot\|$, $f: \mathbb{E}\rightarrow\mathbb{R}$ is a continuously differentiable convex function with a Lipschitz continuous gradient, and $P:\mathbb{E}\to(-\infty, \infty]$ is a proper closed convex function. Problem \eqref{mainpro} arises in various areas such as machine learning, data science, and image/signal processing, and has been extensively studied in the literature, see, e.g., \cite{bt2009a,cd2015convergence,jst2012inexact,n2013gradient,srb2011convergence,vsbv2013accelerated}. 

Many modern applications of \eqref{mainpro} involve extremely large-scale decision variables (often in the millions), making first-order methods particularly appealing due to their computational simplicity and well-established convergence properties. Among such methods, the proximal gradient (PG) method \cite{bt2009a,cp2011proximal,cw2005signal,fm1981a,lm1979splitting} is one of the most prominent. Given an estimate Lipschitz constant $L$ of $\nabla f$, the PG method generates a sequence $\{\bm{x}^k\}$ using the scheme:
\begin{equation}\label{pgscheme}
\bm{x}^{k+1}=\mathop{\mathrm{argmin}}\limits_{\bm{x}\in\mathbb{E}} \left\{P(\bm{x}) + \langle \nabla f(\bm{x}^k), \,\bm{x}-\bm{x}^k\rangle + \frac{L}{2}\|\bm{x}-\bm{x}^k\|^2\right\}.
\end{equation}
It is well known that the PG method achieves the convergence rate $F(\bm{x}^k) - \min\limits_{\bm{x} \in \mathbb{E}}\{F(\bm{x})\}=\mathcal{O}(1/k)$; see, e.g., \cite{bt2009a,n2005smooth}. To improve this rate, Nesterov-type extrapolation can be incorporated, leading to the accelerated proximal gradient (APG) method, which attains the $\mathcal{O}(1/k^2)$ convergence rate for the objective function values. A typical APG scheme for solving problem \eqref{mainpro} takes the form
\begin{align}
\bm{y}^k &= \bm{x}^k + \beta_k (\bm{x}^k - \bm{x}^{k-1}), \label{pgescheme1}
\\
\bm{x}^{k+1} &= \mathop{\mathrm{argmin}}\limits_{\bm{x}\in\mathbb{E}}
\left\{ P(\bm{x}) + \langle \nabla f(\bm{y}^k), \,\bm{x} - \bm{y}^k\rangle + \frac{L}{2}\|\bm{x} - \bm{y}^k\|^2 \right\}, \label{pgescheme2}
\end{align}
where $\{\beta_k\}$ is a sequence of extrapolation parameters satisfying suitable conditions. The fast iterative soft-thresholding algorithm (FISTA) \cite{bt2009a} is a widely used instance of this framework, where $\{\beta_k\}$ are updated as $\beta_k=(t_{k-1}-1)/t_k$ with $t_{k}=\left(1+\sqrt{1+4t_{k-1}^2}\right)/2$ and $t_{-1}=t_0=1$. Although FISTA enjoys fast convergence in the objective function values, the convergence of the iterate sequence $\{\bm{x}^k\}$ itself remained unclear until the work of Chambolle and Dossal \cite{cd2015convergence}, who established iterate convergence by modifying the extrapolation parameters as $\beta_k = (t_{k-1}-1)/t_k$ with $t_k=\frac{k+\alpha-1}{\alpha-1}$ for $\alpha>3$, while preserving the $\mathcal{O}(1/k^2)$ convergence rate for the objective function values.

The theoretical analyses of the PG and APG methods typically assume that the subproblem \eqref{pgscheme} (or \eqref{pgescheme2}) is solved \textit{exactly}. However, in practice, exact solutions are often computationally expensive or even unattainable, especially when $P$ is complex or when the problem dimension is large. This limitation has motivated extensive research on inexact variants of the PG and APG methods, where the subproblem is solved only \textit{inexactly} while retaining provable convergence guarantees; see, e.g., \cite{accr2018inertial,acpr2018fast,ap2016rate,ad2015stability,bgk2023on,jst2012inexact,mm2019inexact,srb2011convergence,vgds2023accelerated,vsbv2013accelerated}. In the following, we introduce several representative inexact variants of the APG method.

Early work by Schmidt et al. \cite{srb2011convergence} allowed inexactness in evaluations of both the gradient and the proximity operator. Specifically, at the $k$th iteration, an approximate solution $\bm{x}^{k+1}$ to the subproblem \eqref{pgescheme2} is permitted, provided that it satisfies
\begin{equation}\label{srb2011-cond}
\widetilde{h}_k(\bm{x}^{k+1}) - \min_{\bm{x} \in \mathbb{E}} \left\{\widetilde{h}_k(\bm{x})\right\} \leq \varepsilon_k,
\end{equation}
where
\begin{equation}\label{eq:def_htilde}
\widetilde{h}_k(\bm{x}) := P(\bm{x}) + \langle \nabla f(\bm{y}^k) + \Delta^k, \,\bm{x} - \bm{y}^k\rangle + \frac{L}{2}\|\bm{x} - \bm{y}^k\|^2.
\end{equation}
Here, $\Delta^k$ represents the error in the gradient evaluation, while $\varepsilon_k$ indicates that $\bm{x}^{k+1}$ is an $\varepsilon_k$-minimizer of the objective function associated with the proximal mapping of $P$. While the error criterion \eqref{srb2011-cond} is difficult to verify directly due to the unknown minimum value, it can often be assessed in practice using a duality gap. At nearly the same time, Villa et al. \cite{vsbv2013accelerated} proposed an alternative criterion based on the $\varepsilon$-subdifferential of $P$. Specifically, at the $k$th iteration, the iterate $\bm{x}^{k+1}$ is required to satisfy
\begin{equation}\label{vsbv2013-cond}
0 \in \partial_{\varepsilon_k} P(\bm{x}^{k+1}) + \nabla f(\bm{y}^k) + L(\bm{x}^{k+1}-\bm{y}^k),
\end{equation}
where $\varepsilon_k$ quantifies the allowable error in approximating the subdifferential of $P$. While the $\varepsilon$-subdifferential criterion \eqref{vsbv2013-cond} is stronger than the $\varepsilon$-minimizer criterion \eqref{srb2011-cond} (as any point satisfying \eqref{vsbv2013-cond} can readily satisfy \eqref{srb2011-cond} with $\Delta^k=0$), this stronger criterion can offer a better (weaker) dependence on the error decay, as discussed in \cite[Section 1.2]{vsbv2013accelerated}. Nevertheless, \textit{neither} Schmidt et al. \cite{srb2011convergence} \textit{nor} Villa et al. \cite{vsbv2013accelerated} studied the iterate convergence for their inexact APG methods. Later, Aujol and Dossal \cite{ad2015stability} extended the work of Chambolle and Dossal \cite{cd2015convergence} to the inexact setting, and provided an analysis of the iterate convergence. They introduced two distinct approximation criteria for solving the subproblem \eqref{pgescheme2}. Specifically, at the $k$th iteration, their framework allows for an approximate solution $\bm{x}^{k+1}$ that satisfies either
\begin{equation}\label{ad2015-cond}
(a)~~ 0 \in \partial_{\varepsilon_k} \widetilde{h}_k(\bm{x}^{k+1}),
\quad\mbox{or}\quad
(b)~~ \Delta^{k} \in \partial_{\varepsilon_k}P(\bm{x}^{k+1})
+ \nabla f(\bm{y}^k) + L\big(\bm{x}^{k+1}-\bm{y}^k\big),
\end{equation}
where $\widetilde{h}_k$ is defined in \eqref{eq:def_htilde}. Note that criterion (a) in \eqref{ad2015-cond} is equivalent to \eqref{srb2011-cond}, while criterion (b) in \eqref{ad2015-cond} generalizes \eqref{vsbv2013-cond} by incorporating an additional error in the gradient evaluation. With special choices of extrapolation parameters, their analysis further establishes an improved $o(1/k^2)$ convergence rate for the objective function values.

All the error criteria discussed so far are of the absolute type, meaning they require one or more sequences of tolerance parameters that satisfy suitable summable conditions in advance. Recently, Bello-Cruz et al. \cite{bgk2023on} proposed a relative-type error criterion for the APG method. Specifically, at the $k$th iteration, they allow finding a triple $(\bm{x}^{k+1},\, \Delta^{k},\, \varepsilon_{k})$ that satisfies
\begin{equation}\label{bgk2023on-cond}
\begin{aligned}
&\Delta^{k} \in \partial_{\varepsilon_k} P(\bm{x}^{k+1}) + \nabla f(\bm{y}^k) + \frac{L}{\tau}(\bm{x}^{k+1} - \bm{y}^{k}), \\
&\mathrm{with}~~ \|\tau \Delta^{k}\|^2 + 2\tau\varepsilon_k L \leq L\left(\left(1-\tau\right)L - \alpha \tau\right) \| \bm{x}^{k+1} - \bm{y}^{k}\|^2, \\
\end{aligned}
\end{equation}
where $\tau\in(0,1]$ and $\alpha\in\left[0, \frac{(1-\tau)L}{\tau}\right]$. Compared to the absolute-type error criterion (e.g., \eqref{srb2011-cond} and \eqref{vsbv2013-cond}), the relative-type error criterion \eqref{bgk2023on-cond} eliminates the need to predefine summable sequences of tolerance parameters and is typically more tuning-friendly. However, this advantage comes at the cost of increased verification complexity and computational overhead, as well as a potentially shorter step size $\tau/L$, which may in turn degrade practical performance. More inexact APG methods for solving problem \eqref{mainpro} under different scenarios can be found in \cite{accr2018inertial,acpr2018fast,ap2016rate,brr2018inertial,jfzsc2025inexact,mm2019inexact,vgds2023accelerated}.

Despite their differences, a common feature of all the aforementioned error criteria is that they require an approximate iterate $\bm{x}^{k+1}$ to remain \textit{feasible}, i.e., to lie in the domain of $P$, at every iteration. This feasibility requirement can be overly restrictive in practice, particularly for constrained problems with complicated feasible sets. Indeed, enforcing feasibility at every iteration may require expensive projections or auxiliary optimization procedures, which can significantly increase the computational cost. To address this issue, Jiang et al. \cite{jst2012inexact} developed a specialized error criterion for the APG method to solve the linearly constrained convex semidefinite programming problem, which permits controlled infeasibility with respect to linear constraints. However, their framework is problem-specific: it does not accommodate general nonsmooth regularizers, relies on strong assumptions on the dual solutions of the subproblems, and does not establish either the iterate convergence or an $o(1/k^2)$ convergence rate for the objective function values.

Motivated by these limitations and inspired by recent advances in inexact mechanisms studied in \cite{clty2023an,yt2022bregman,yt2025inexact}, in this paper, we revisit inexact APG methods and aim to develop a more flexible, practically implementable, and computationally economical inexact variant. Specifically, we introduce a \underline{s}hadow-\underline{p}oint enhanced \underline{in}exact \underline{APG} method, termed SpinAPG. The key novelty of SpinAPG lies in its error criterion, which is constructed using two distinct points: a possibly \textit{infeasible} iterate $\bm{x}^{k+1}$ and a conceptually defined \textit{feasible} shadow point $\widetilde{\bm{x}}^{k+1}$, the latter of which does not require explicit computation. This design allows us to decouple feasibility enforcement from the main APG iteration, thereby potentially avoiding explicit projections that are commonly required in existing inexact APG methods. Moreover, the proposed error criterion also accommodates errors in evaluations of both the gradient and the proximity operator, aligning with several existing error criteria. As a result, SpinAPG is highly adaptable to various types of errors that arise when solving subproblems inexactly, and it thus significantly enhances the practical applicability of the APG method, particularly for large-scale constrained problems. To highlight the distinctions among existing inexact APG methods, we provide a brief comparison in Table \ref{tab:cmp-algos}.

\begin{table}[ht]
\caption{Comparisons of inexact APG methods.}\label{tab:cmp-algos}
\centering \tabcolsep 2.5pt
\vspace{-1mm}
\scalebox{0.99}{
\begin{tabular}{cccccccc}
\toprule
\multirow{2}{*}{\texttt{Reference}}      & \multicolumn{2}{c}{\texttt{Allow Error}} & \multirow{2}{*}{\shortstack{\texttt{Allow} \\[2pt] \texttt{Infeas.}}} & \multicolumn{2}{c}{\texttt{Fun. Val. Rate}} & \multirow{2}{*}{\shortstack{\texttt{Iterate} \\[2pt] \texttt{Conv.}}} & \multirow{2}{*}{\texttt{Type}} \\
\cmidrule(lr){2-3} \cmidrule(lr){5-6}
& \texttt{Grad.}  & \texttt{Prox. Operator}  &  & $\mathcal{O}(1/k^2)$ & $o(1/k^2)$ &&  \\
\midrule
Schmidt et al. \cite{srb2011convergence}  & \ding{51} & $\varepsilon$-min. & \ding{55} & \ding{51} & \ding{55} & \ding{55} & \texttt{Abs.} \\
Villa et al. \cite{vsbv2013accelerated}  & \ding{55} & $\varepsilon$-subdiff.  & \ding{55} & \ding{51} & \ding{55} & \ding{55} & \texttt{Abs.} \\
Aujol\,\&\,Dossal \cite{ad2015stability} & \ding{51} & $\varepsilon$-min. or $\varepsilon$-subdiff.  & \ding{55} & \ding{51} & \ding{51} & \ding{51} & \texttt{Abs.} \\
Bello-Cruz et al. \cite{bgk2023on}  & \ding{51} & $\varepsilon$-subdiff. & \ding{55} & \ding{51} & \ding{55} & \ding{55} & \texttt{Rel.} \\
Jiang et al. \cite{jst2012inexact}  & \ding{51} & $\varepsilon$-subdiff. & \ding{51} & \ding{51} & \ding{55} & \ding{55} & \texttt{Abs.} \\
SpinAPG [this work]  & \ding{51} & $\varepsilon$-subdiff./two points & \ding{51} & \ding{51} & \ding{51} & \ding{51} & \texttt{Abs.}  \\
\bottomrule
\end{tabular}}
\end{table}

The key contributions of this paper are summarized as follows. 
\begin{itemize}[leftmargin=1.15cm]

\item[{\bf 1.}] We develop SpinAPG, a shadow-point enhanced inexact APG method, for solving \eqref{mainpro}. In contrast to existing inexact APG methods \cite{ad2015stability,bgk2023on,jst2012inexact,srb2011convergence,vsbv2013accelerated}, SpinAPG accommodates errors in both gradient evaluations and proximity operator computations, while offering the potential to address feasibility challenges by avoiding explicit projections for computing feasible iterates. This design enhances the flexibility and computational efficiency of inexact APG methods, particularly for problems with complex feasible regions.

\item[{\bf 2.}] Under suitable summability conditions on the error terms, we establish the ${\cal O}(1/k^2)$ convergence rate for the objective function values. Moreover, for specific choices of extrapolation parameters, this rate can be further improved to $o(1/k^2)$ and the iterate sequence is shown to converge. These results complement and extend existing convergence analyses of inexact APG methods by demonstrating that all desirable convergence properties can be preserved even when subproblems are solved inexactly and feasibility is relaxed via the shadow-point mechanism, provided the errors are properly controlled.

\item[{\bf 3.}] We conduct numerical experiments on sparse quadratic programming problems to demonstrate the computational advantages of SpinAPG, which reduces overhead by avoiding explicit computations of feasible points. These results highlight the value of developing inexact APG variants that permit controlled infeasibility.

\end{itemize}

\vspace{1mm}
The rest of this paper is organized as follows. We describe SpinAPG for solving problem \eqref{mainpro} in Section \ref{sec-SpinAPG} and establish the convergence results in Section \ref{sec-convana}. Some numerical experiments are conducted in Section \ref{sec-num}, with conclusions given in Section \ref{seccon}.

\vspace{2mm}
\noindent\textbf{Notation}.
We use lowercase, bold lowercase, and uppercase letters to denote scalars, vectors, and matrices, respectively. The sets $\mathbb{R}$, $\mathbb{R}^n$, $\mathbb{R}^n_+$ ($\mathbb{R}^n_{-}$) and $\mathbb{R}^{m\times n}$ represent real numbers, $n$-dimensional real vectors, nonnegative (nonpositive) $n$-dimensional real vectors, and $m\times n$ real matrices. For $\bm{x}\in\mathbb{R}^n$, $\|\bm{x}\|$ denotes its $\ell_2$ norm. For $A\in\mathbb{R}^{m\times n}$, $\|A\|_F$ denotes its Frobenius norm. 
%\blue{For a linear operator $\mathcal{H}:\mathbb{R}^{m\times n}\rightarrow \mathbb{R}^{m\times n}$, the operator norm $\|\mathcal{H}\|$ is defined as $\|\mathcal{H}\| = \sup_{X\in \mathbb{R}^{m\times n}} \|\mathcal{H}(X)\|_F/\|X\|_F$.} 
For a closed convex set $\mathcal{X}\subseteq\mathbb{E}$, its indicator function $\delta_{\mathcal{X}}$ is defined as $0$ if $\bm{x}\in\mathcal{X}$ and $+\infty$ otherwise. For a proper closed convex function $f: \mathbb{E} \rightarrow (-\infty, \infty]$ and a given $\varepsilon \geq 0$, the $\varepsilon$-subdifferential of $f$ at $\bm{x}\in{\rm dom}\,f$ is $\partial_{\varepsilon} f(\bm{x}):=\{\bm{d}\in\mathbb{E}: f(\bm{y}) \geq f(\bm{x}) + \langle \bm{d}, \,\bm{y}-\bm{x} \rangle - \varepsilon, ~\forall\,\bm{y}\in\mathbb{E}\}$. When $\varepsilon=0$, $\partial_{\varepsilon} f$ reduces to the subdifferential $\partial f$. For any $\nu>0$, the Moreau envelope of $\nu f$ at $\bm{x}$ is defined by $\mathtt{M}_{\nu f}(\bm{x}) := \min_{\bm{y}} \big\{f(\bm{y}) + \frac{1}{2\nu}\|\bm{y} - \bm{x}\|^2\big\}$, and the proximal mapping of $\nu f$ at $\bm{x}$ is defined by $\mathtt{prox}_{\nu f}(\bm{x}) := \arg\min_{\bm{y}} \big\{f(\bm{y}) + \frac{1}{2\nu}\|\bm{y} - \bm{x}\|^2\big\}$.

%%%%%%%%%%%%%%%%%%%%%%%%%%%%%%%%%%%%%%%%%%%%%%%%%%%%%%%%%%%%%%
\section{A shadow-point enhanced inexact APG method}\label{sec-SpinAPG}

In this section, we build upon the APG method to develop a novel shadow-point enhanced inexact variant, termed SpinAPG, for solving \eqref{mainpro}. To begin, we specify conditions on a parameter sequence $\{\theta_k\}_{k=-1}^{\infty}$, which is used to construct the extrapolation parameter in the algorithm. Specifically, we consider a sequence $\{\theta_k\}_{k=-1}^{\infty}\subseteq(0,1]$ with $\theta_{-1}=\theta_0=1$, such that for all $k\geq1$:
\begin{numcases}{}
\theta_k\leq\frac{\alpha-1}{k+\alpha-1}, ~~\alpha\geq 3, \label{condtheta1} \\
\frac{1-\theta_{k}}{\theta_{k}^{ 2 }} \leq \frac{1}{\theta_{k-1}^{ 2 }}. \label{condtheta}
\end{numcases}
These conditions unify a broad class of extrapolation parameter setting rules in the literature; see, e.g., \cite{cd2015convergence,hrx2018accelerated,t2008on}. Many APG variants, either implicitly or explicitly, adopt such conditions to ensure desirable convergence properties; see, e.g., \cite{ad2015stability,bpr2016variable,brr2018inertial,cd2015convergence,t2008on}. In particular, when \eqref{condtheta} holds as an equality, $\{\theta_k\}$ satisfies the recurrence relation: $\theta_{k} = \frac{1}{2}\left(\sqrt{\theta_{k-1}^4+4\theta_{k-1}^2}-\theta_{k-1}^2\right)$ for all $k\geq1$. This is the classical choice introduced in Nesterov's seminal works \cite{n1983a,n1988on,n2003introductory} and widely adopted in subsequent developments, e.g., \cite{at2006interior,bt2009a,jst2012inexact,t2010approximation}.
Under this choice, an inductive argument shows that $\theta_k\leq\frac{2}{k+2}$ for all $k\geq0$, and as $k\to\infty$, $\theta_k = \frac{2}{k+2}+o(1)$. A more recent alternative, proposed by Chambolle and Dossal \cite{cd2015convergence}, sets $\theta_k = \frac{\alpha-1}{k+\alpha-1}\,(\forall\,k\geq1)$ for some $\alpha\geq3$, which satisfies \eqref{condtheta}. Based on this setting, they established the convergence of iterates when $\alpha>3$. Independently, Attouch et al. \cite{acpr2018fast} obtained a similar result using different arguments. Later, Attouch and Peypouquet \cite{ap2016rate} further showed that $\alpha>3$ yields the improved $o(1/k^2)$ convergence rate for the objective function values.

With these preparations, We are now ready to present SpinAPG in Algorithm \ref{algo-SpinAPG}.

\begin{algorithm}[ht]
\caption{A shadow-point enhanced inexact APG method (SpinAPG) for \eqref{mainpro}}\label{algo-SpinAPG}
\textbf{Input:} Let $\{\eta_k\}_{k=0}^{\infty}$, $\{\mu_k\}_{k=0}^{\infty}$ and $\{\nu_k\}_{k=0}^{\infty}$ be three sequences of nonnegative scalars. Choose $\bm{x}^{-1} = \bm{x}^0\in\mathbb{E}$ arbitrarily. Set $\theta_{-1}=\theta_0=1$ and $k=0$.  \\
\textbf{while} a termination criterion is not met, \textbf{do} %\vspace{-2mm}
\begin{itemize}[leftmargin=2cm]
\item[\textbf{Step 1}.] Compute $\bm{y}^k = \bm{x}^k + \theta_k(\theta_{k-1}^{-1}-1)(\bm{x}^k-\bm{x}^{k-1})$.

\vspace{0.5mm}
\item[\textbf{Step 2}.] Find a pair $(\bm{x}^{k+1}, \,\widetilde{\bm{x}}^{k+1})$ and an error pair $(\Delta^k,\,\varepsilon_k)$ by approximately solving
    \begin{equation}\label{subpro-iAPG}
    \min\limits_{\bm{x}\in \mathbb{E}}~ \Big\{P(\bm{x})
    + \langle \nabla f(\bm{y}^k), \,\bm{x}-\bm{y}^k\rangle
    + \frac{L}{2}\|\bm{x}-\bm{y}^k\|^2\Big\},
    \end{equation}
    such that $\bm{x}^{k+1} \in \mathbb{E}$, $\widetilde{\bm{x}}^{k+1} \in \mathrm{dom}\,P$, and
    \begin{equation}\label{inexcond-iAPG}
    \begin{aligned}
    &\quad \Delta^{k} \in \partial_{\varepsilon_k}P(\widetilde{\bm{x}}^{k+1})
    + \nabla f(\bm{y}^k) + L\big(\bm{x}^{k+1}-\bm{y}^k\big) \\[3pt]
    &~~\mathrm{with}~~\|\Delta^{k}\| \leq \eta_k,
    ~~\varepsilon_k\leq\nu_k,
    ~~\|\widetilde{\bm{x}}^{k+1}-\bm{x}^{k+1}\| \leq \mu_{k}.
    \end{aligned}
    \end{equation}

\item[\textbf{Step 3}.] Choose $\theta_{k+1}\in(0,1]$ satisfying conditions \eqref{condtheta1} and \eqref{condtheta}.

\vspace{0.5mm}
\item[\textbf{Step 4}.] Set $k=k+1$. %and go to \textbf{Step 1}.
\end{itemize}
\textbf{end while}  \\
\textbf{Output}: $\bm{x}^{k}$ %\vspace{0.5mm}
\end{algorithm}

%%%%%%%%%%%%%%%%%%%%%%%%%%%%%%%%%%%%%%%%%%%%%%%%%%%%
%\subsection{Shadow-point enhanced error criterion}
\subsection{Shadow-point mechanism for inexactness}

At each iteration, our inexact framework permits an \textit{approximate infeasible} solution of the subproblem \eqref{subpro-iAPG} under the error criterion \eqref{inexcond-iAPG}. Since the subproblem \eqref{subpro-iAPG} admits a unique solution $\bm{x}^{k,*}$ due to the strong convexity, criterion \eqref{inexcond-iAPG} is always attainable: it holds trivially at $\bm{x}^{k+1}=\widetilde{\bm{x}}^{k+1}=\bm{x}^{k,*}$ with $\|\Delta^k\|=\varepsilon_k=0$. In particular, setting $\eta_k=\nu_k=\mu_k=0$ in \eqref{inexcond-iAPG} forces $\bm{x}^{k+1}$ (and thus $\widetilde{\bm{x}}^{k+1}$) to coincide with the exact solution of the subproblem, in which case SpinAPG reduces to the classical exact APG method studied in \cite{bt2009a,cd2015convergence}. Moreover, one can observe that the error criterion \eqref{inexcond-iAPG} is of a two-point type, involving $\bm{x}^{k+1}$ and $\widetilde{\bm{x}}^{k+1}$. Since $\widetilde{\bm{x}}^{k+1}$ is introduced solely for constructing criterion \eqref{inexcond-iAPG} without requiring explicit computation, and does not participate in the iterate updates, we refer to it as a \textit{shadow point}. %which motivates the term \textit{shadow-point enhanced inexact APG method}.

This two-point inexact mechanism is inspired by recent works on the inexact Bregman proximal point algorithm \cite{clty2023an,yt2022bregman} and the inexact Bregman proximal gradient method \cite{yt2025inexact}. However, SpinAPG incorporates several notable design differences that substantially improve its practical efficiency. First, in the inertial method of \cite{yt2025inexact}, the ``shadow" point $\widetilde{\bm{x}}^{k+1}$ explicitly participates in the iterate updates and therefore must be computed at each iteration, leading to additional computational overhead. In contrast, SpinAPG uses the shadow point solely as a theoretical tool for formulating and verifying the error criterion, without requiring its explicit computation during the updates, thereby leading to a more economical implementation. Second, while inertial variants were also developed in those works, they rely on a substantially more involved acceleration framework designed to accommodate Bregman proximal terms, which differs fundamentally from the classical and more streamlined acceleration scheme \eqref{pgescheme1}–\eqref{pgescheme2} considered in the present work.

% While inertial variants were also developed in those works, they rely on a substantially more involved acceleration framework designed to accommodate Bregman proximal terms, which differs fundamentally from the classical and simpler acceleration scheme \eqref{pgescheme1}–\eqref{pgescheme2} considered in the present work. Moreover, in the inertial method of \cite{yt2025inexact}, the ``shadow" point $\widetilde{\bm{x}}^{k+1}$ explicitly enters the iterate updates and therefore must be computed at each iteration, leading to additional computational overhead. In contrast, SpinAPG uses the shadow point solely for the formulation and verification of the error criterion, making it more amenable to economical implementation.

At first glance, the shadow-point enhanced error criterion \eqref{inexcond-iAPG} may appear unconventional. Nevertheless, it provides a flexible and comprehensive framework for handling approximate solutions. Specifically, it captures inexactness through three key components: the error term $\Delta^{k}$, which appears on the left-hand-side of the optimality condition; the $\varepsilon_k$-subdifferential $\partial_{\varepsilon_k}P$, which serves as an approximation of the exact subdifferential $\partial P$; and the deviation $\|\widetilde{\bm{x}}^{k+1}-\bm{x}^{k+1}\|$, which measures the discrepancy between the computed (possibly infeasible) iterate and its associated feasible shadow-point. This formulation yields a highly adaptable inexact framework capable of accommodating various types of errors in solving the subproblem. Moreover, criterion \eqref{inexcond-iAPG} %implicitly or explicitly 
generalizes several existing error criteria. For example, the works \cite{acpr2018fast,ap2016rate} only consider the error term $\Delta^k$, corresponding to \eqref{inexcond-iAPG} with $\nu_k\equiv\mu_k\equiv0$. In contrast, \cite{vsbv2013accelerated} focuses solely on the approximate computation of the subdifferential of $P$, corresponding to \eqref{inexcond-iAPG} with $\eta_k\equiv\mu_k\equiv0$. Several other works (e.g., \cite{ad2015stability,bgk2023on,brr2018inertial,jst2012inexact}) incorporate both types of inexactness. Up to a scaling matrix, these approaches can be viewed as special cases of \eqref{inexcond-iAPG} with $\mu_k\equiv0$. Therefore, all examples used to verify the error criteria in these existing works remain valid within our inexact framework.

%%%%%%%%%%%%%%%%%%%%%%%%%%%%%%%%%%%%%%%%%%%%%%%%%%%%
\subsection{Overcoming feasibility challenges via shadow-points}\label{sec-example}

A key advantage of the proposed SpinAPG over existing inexact APG methods (e.g., \cite{acpr2018fast,ap2016rate,ad2015stability,bgk2023on,brr2018inertial,srb2011convergence,vsbv2013accelerated}) is its potential to address feasibility challenges when dealing
with problems with complex feasible regions. Unlike traditional inexact APG methods, which typically require all iterates to remain strictly within the feasible region $\mathrm{dom}\,P$, SpinAPG permits the computed iterate $\bm{x}^{k+1}$ to be potentially \textit{infeasible}. To ensure convergence, an associated shadow point $\widetilde{\bm{x}}^{k+1}\in \mathrm{dom}\,P$ is introduced as a \textit{feasible} reference point in the error criterion and guides the evolution of $\bm{x}^{k+1}$. Importantly, under suitable error bound conditions on the feasible set, this shadow point does not need to be computed explicitly. As a result, SpinAPG can substantially reduce the computational overhead associated with enforcing feasibility, a task that can be particularly costly for large-scale constrained problems. The following example illustrates how this shadow-point enhanced inexact mechanism achieves a favorable balance between feasibility requirement and computational efficiency in practice.

Let $J\subseteq\{1,\ldots,n\}$ be an index set, with $|J|$ denoting its cardinality and $J^\complement = \{1,\ldots,n\}\backslash J$ denoting its complementarity. For a vector $\bm{x}\in\mathbb{R}^n$, let $\bm{x}_{J}\in\mathbb{R}^{|J|}$ and $\bm{x}_{J^\complement}\in\mathbb{R}^{|J^\complement|}$ denote the corresponding subvectors. Consider the convex function
\begin{equation*}
g(\bm{x}) := r(\bm{x}_J) + \delta_{\{\bm{x}_{J^\complement}\ge 0\}}(\bm{x}),
\end{equation*}
where $r:\mathbb{R}^{|J|}\to\mathbb{R}$ is a continuous convex function and $\delta_{\{\bm{x}_{J^\complement}\ge 0\}}$ is the indicator function of the set $\{\bm{x}\in\mathbb{R}^n:\bm{x}_{J^\complement}\ge 0\}$. Assume that, for any $\nu>0$, the proximal mapping of $\nu g$ is easy to compute. We then consider the composite regularizer
\begin{equation*}%\label{eq:example_P}
P(\bm{x}) := g(\bm{x}) + \delta_{\{A\bm{x}=\bm{b}\}}(\bm{x}),
\end{equation*}
where $A\in\mathbb{R}^{m\times n}$, $\bm{b}\in \mathbb{R}^m$, and $\delta_{\{A\bm{x}=\bm{b}\}}$ is the indicator function of the set $\{\bm{x}\in\mathbb{R}^n:A\bm{x}=\bm{b}\}$. Problem \eqref{mainpro} with this choice of $P$ covers a broad class of block-regularized problems with affine constraints, with inequality constraints handled via standard slack-variable reformulations. Under this setting, the $k$th APG subproblem \eqref{subpro-iAPG} takes the form
\begin{equation}\label{eq:apg_sub_ex}
\min\limits_{\bm{x}\in\mathbb{R}^n}~ \left\{g(\bm{x})
+ \langle \nabla f(\bm{y}^k), \,\bm{x}-\bm{y}^k\rangle
+ \frac{L}{2}\|\bm{x}-\bm{y}^k\|^2 \ \middle\vert\  A\bm{x}=\bm{b} \right\},
\end{equation}
which can be efficiently solved on its dual side. Specifically, the dual problem of \eqref{eq:apg_sub_ex} can be given (in a minimization form) by
\begin{equation}
\min\limits_{\bm{z}\in \mathbb{R}^m}~~\Psi_k(\bm{z})
:=-\,\texttt{M}_{g/L}\big(A^{\top}\bm{z}/L - \nabla f(\bm{y}^k)/L + \bm{y}^k\big)
+ \frac{1}{2L}\big\|A^{\top}\bm{z}-\nabla f(\bm{y}^k)\big\|^2
+ \langle A\bm{y}^k-\bm{b},\,\bm{z}\rangle,\label{eq:def_dual}
\end{equation}
where $\bm{z}\in\mathbb{R}^m$ is the dual variable. By the property of the Moreau envelope (see, e.g.,
\cite[Proposition 12.29]{bc2011convex}), $\Psi_k$ is convex and continuously differentiable with the gradient:
\begin{equation*}
\nabla \Psi_k(\bm{z}) = A \texttt{prox}_{g/L}\left(A^{\top}\bm{z}/L - \nabla f(\bm{y}^k)/L + \bm{y}^k\right) - \bm{b}.
\end{equation*}
Thus, gradient-type methods, or when $\nabla \Psi_k$ is semi-smooth, the semi-smooth Newton ({\sc Ssn}) method (see, e.g., \cite[Section 3]{lst2020on}), can be employed to solve this dual problem efficiently.

Let $\{\bm{z}^{k,t}\}$ be a convergent dual sequence such that $\nabla \Psi_k(\bm{z}^{k,t})\to0$ as $t\to \infty$, and define the corresponding primal sequence by
\begin{equation}\label{eq:ssn_x}
\bm{x}^{k,t} := \texttt{prox}_{g/L}\left(A^{\top}\bm{z}^{k,t} /L - \nabla f(\bm{y}^k)/L +\bm{y}^k\right),
\end{equation}
which converges to the solution of the subproblem \eqref{eq:apg_sub_ex}. By construction, we know that $\bm{x}_{J^\complement}^{k,t}\geq0$, but $\bm{x}^{k,t}$ may not satisfy the affine constraint $A\bm{x}=\bm{b}$ and hence may be \textit{infeasible} for the subproblem \eqref{eq:apg_sub_ex}, i.e., $\bm{x}^{k,t}$ is not guaranteed to be in $\mathrm{dom}\,P=\{\bm{x}\in \mathbb{R}^n \mid A\bm{x}=\bm{b}, \,\bm{x}_{J^\complement}\geq 0 \}$. Within existing inexact APG frameworks \cite{acpr2018fast,ap2016rate,ad2015stability,bgk2023on,brr2018inertial,srb2011convergence,vsbv2013accelerated}, this typically requires additional projections onto $\mathrm{dom}\,P$ to obtain feasible iterates, which might be computationally expensive. In contrast, our inexact framework is capable of avoiding such projections.

To make this distinction concrete, we next explain how error criteria are verified for SpinAPG and representative inexact APG schemes: iAPG-SLB, an inexact accelerated proximal gradient method proposed by Schmidt et al. \cite{srb2011convergence}; AIFB, an accelerated inexact forward-backward algorithm proposed by Villa et al. \cite{vsbv2013accelerated}; o-iFB, an over-relaxation of the inertial forward-backward algorithm proposed by Aujol and Dossal \cite{ad2015stability}; and I-FISTA, an inexact fast iterative shrinkage/thresholding algorithm with a relative error rule proposed by Bello-Cruz et al. \cite{bgk2023on}.

\textbf{SpinAPG.} Let $\Pi_{\mathrm{dom}\,P}$ be the projection operator onto $\mathrm{dom}\,P$ and $G>0$ be any positive constant. By the Hoffman error bound theorem \cite{h1952on}, for any $\bm{x}^{k,t}$ defined in \eqref{eq:ssn_x}, by setting $\widetilde{\bm{x}}^{k,t}:=\Pi_{\mathrm{dom}\,P}(\bm{x}^{k,t})$, there exists a constant $\kappa>0$ such that
\begin{equation}\label{errorbd}
\|\widetilde{\bm{x}}^{k,t}-\bm{x}^{k,t}\|
=\|\Pi_{\mathrm{dom}\,P}(\bm{x}^{k,t})-\bm{x}^{k,t}\|
\leq \kappa\|A\bm{x}^{k,t}-\bm{b}\|
= \kappa \|\nabla\Psi_k(\bm{z}^{k,t})\|
< G,
\end{equation}
where the last inequality holds for sufficiently large $t$ since $\nabla \Psi_k(\bm{z}^{k,t})\to0$ as $t\to0$. Then, for any $\bm{x}\in\mathrm{dom}\,P$, we have
\begin{equation*}
\begin{aligned}
&\quad P(\widetilde{\bm{x}}^{k,t}) - P(\bm{x})
+ \langle - \nabla f(\bm{y}^k) - L(\bm{x}^{k,t} - \bm{y}^k), \,\bm{x} - \widetilde{\bm{x}}^{k,t}\rangle \\
&= g(\widetilde{\bm{x}}^{k,t}) - g(\bm{x})
+ \langle A^\top \bm{z}^{k,t}- \nabla f(\bm{y}^k)
- L(\bm{x}^{k,t} - \bm{y}^k), \,\bm{x} - \widetilde{\bm{x}}^{k,t}\rangle \\
&\leq  g(\widetilde{\bm{x}}^{k,t}) - g(\bm{x}^{k,t})
+ \langle A^\top \bm{z}^{k,t}- \nabla f(\bm{y}^k) - L(\bm{x}^{k,t} - \bm{y}^k), \,\bm{x}^{k,t} - \widetilde{\bm{x}}^{k,t}\rangle \\
&= r(\widetilde{\bm{x}}^{k,t}_J) - r(\bm{x}^{k,t}_J)
+ \langle A^\top \bm{z}^{k,t} - \nabla f(\bm{y}^k) - L(\bm{x}^{k,t} - \bm{y}^k), \,\bm{x}^{k,t} - \widetilde{\bm{x}}^{k,t} \rangle\\
&\leq \kappa \big(\ell_r^G + \|A^\top \bm{z}^{k,t} - \nabla f(\bm{y}^k) - L(\bm{x}^{k,t} - \bm{y}^k)\| \big)\|\nabla\Psi(\bm{z}^{k,t})\|,
\end{aligned}
\end{equation*}
where the first equality follows from $A\bm{x}=A\widetilde{\bm{x}}^{k,t}=\bm{b}$, the first inequality follows from the setting of $\bm{x}^{k,t}$ in \eqref{eq:ssn_x}, the second equality holds since $\widetilde{\bm{x}}^{k,t}_{J^\complement},\,\bm{x}^{k,t}_{J^\complement}\geq 0$, and the last inequality follows from \eqref{errorbd} and the fact that the convex function $r$ is Lipschitz continuous with a Lipschitz constant $\ell_r^G$ on a bounded set of diameter $G$ containing $\widetilde{\bm{x}}^{k,t}$ and $\bm{x}^{k,t}$. Consequently, we obtain the following approximate optimality condition
\begin{equation}\label{ssn_inclusion}
0 \in \partial_{\varepsilon_{k,t}}P(\widetilde{\bm{x}}^{k,t})
+ \nabla f(\bm{y}^k) + L\big(\bm{x}^{k,t}-\bm{y}^k\big),
\end{equation}
with $\varepsilon_{k,t} := \kappa \big(\ell_r^G + \|A^\top \bm{z}^{k,t} - \nabla f(\bm{y}^k) - L(\bm{x}^{k,t} - \bm{y}^k)\| \big)\|\nabla\Psi(\bm{z}^{k,t})\|$. Combining \eqref{ssn_inclusion} with \eqref{errorbd}, we see that the error criterion \eqref{inexcond-iAPG} is verifiable at $(\bm{x}^{k,t},\,\widetilde{\bm{x}}^{k,t})\in \mathbb{R}^n \times \mathrm{dom}\,P$ and holds whenever
\begin{equation}\label{checkingcond}
\left\{\begin{aligned}
\varepsilon_{k,t}
= \kappa \big(\ell_r^G + \|A^\top \bm{z}^{k,t} - \nabla f(\bm{y}^k) - L(\bm{x}^{k,t} - \bm{y}^k)\| \big)\|\nabla\Psi(\bm{z}^{k,t})\|
&\leq \nu_k:=\kappa(\ell_r^G+1)\tilde{\nu}_k,  \\[3pt]
\|\widetilde{\bm{x}}^{k,t}-\bm{x}^{k,t}\|
\leq \kappa \|\nabla\Psi_k(\bm{z}^{k,t})\|
&\leq \mu_k:=\kappa\min\{\tilde{\mu}_k, \,G\},
\end{aligned}\right.
\end{equation}
for prescribed tolerances $\tilde{\nu}_k,\,\tilde{\mu}_k>0$. In practical implementations, one may choose a sufficiently large constant $G>0$, in which case it suffices to verify
\begin{equation}\label{SpinAPG_prac_cond}
\left\{\begin{aligned}
&\|\nabla\Psi_k(\bm{z}^{k,t})\| \leq \min\left\{\tilde{\nu}_k,\,\tilde{\mu}_k\right\}, \\[3pt]
&\|A^\top \bm{z}^{k,t} - \nabla f(\bm{y}^k) - L(\bm{x}^{k,t} - \bm{y}^k)\|\|\nabla\Psi(\bm{z}^{k,t})\|
\leq \tilde{\nu}_k,
\end{aligned}\right.
\end{equation}
which readily implies \eqref{checkingcond}.

\textbf{iAPG-SLB} \cite{srb2011convergence} / \textbf{AIFB} \cite{vsbv2013accelerated}. To verify their error criteria \eqref{srb2011-cond} and \eqref{vsbv2013-cond}, one can follow \cite{srb2011convergence,vsbv2013accelerated} to compute the duality gap when solving the subproblem, and terminate the subsolver once the duality gap falls below a given tolerance. Specifically, at the $k$th iteration, the subsolver in iAPG-SLB and AIFB can be terminated when
\begin{equation*}%\label{srb2011-cond-general}
g(\widetilde{\bm{x}}^{k,t})
+ \langle \nabla f(\bm{y}^k), \,\widetilde{\bm{x}}^{k,t}-\bm{y}^k\rangle
+ \frac{L}{2}\|\widetilde{\bm{x}}^{k,t}-\bm{y}^k\|^2 + \Psi_k(\bm{z}^{k,t})
\leq \varepsilon_k.
\end{equation*}
It turns out that iAPG-SLB and AIFB share the same practical implementation for verifying their error criteria. Moreover, the verification of either \eqref{srb2011-cond} or \eqref{vsbv2013-cond} requires explicitly computing a \textit{feasible} point $\widetilde{\bm{x}}^{k,t}$, which also serves as the next proximal point.

\textbf{o-iFB} \cite{ad2015stability}. From \eqref{ssn_inclusion}, we further have that
\begin{equation}\label{oiFB-cond-general}
\Delta^{k,t}:=L(\widetilde{\bm{x}}^{k,t}-\bm{x}^{k,t})
\in\partial_{\varepsilon_{k,t}}P(\widetilde{\bm{x}}^{k,t})
+ \nabla f(\bm{y}^k) + L\big(\widetilde{\bm{x}}^{k,t}-\bm{y}^k\big).
\end{equation}
This relation implies that the error criterion (b) in \eqref{ad2015-cond} is verifiable at $\widetilde{\bm{x}}^{k,t}$ with a pair of errors $(\Delta^{k,t},\,\varepsilon_{k,t})$. Thus, one can varify the same inequalities as in \eqref{SpinAPG_prac_cond} to terminate the subsolver at the $k$th iteration in o-iFB. However, a \textit{feasible} point $\widetilde{\bm{x}}^{k,t}$ must be explicitly computed, which is used as the next proximal point.

\textbf{I-FISTA} \cite{bgk2023on}. First, we note that I-FISTA shares the same subproblem structure as \eqref{eq:apg_sub_ex}, but introduces an additional parameter $\tau\in(0,1)$ to enable inexact minimization under a relative-type error criterion. Specifically, at the $k$th iteration, its subproblem takes the form
\begin{equation*}
\min\limits_{\bm{x}\in\mathbb{R}^n}~ \left\{g(\bm{x})
+ \langle \nabla f(\bm{y}^k), \,\bm{x}-\bm{y}^k\rangle
+ \frac{L}{2\tau}\|\bm{x}-\bm{y}^k\|^2 \ \middle\vert\  A\bm{x}=\bm{b} \right\}.
\end{equation*}
Clearly, it can be efficiently handled on the dual side using the same procedure described above. When a dual-based subsolver produces an approximate solution $\bm{x}^{k,t}:=\texttt{prox}_{\tau g/L}\big(\tau A^{\top}\bm{z}^{k,t}/L - \tau\nabla f(\bm{y}^k)/L +\bm{y}^k\big)$, we use $\Pi_{\mathrm{dom}\,P}$ to obtain a feasible point by setting $\widetilde{\bm{x}}^{k,t}:=\Pi_{\mathrm{dom}\,P}(\bm{x}^{k,t})$. Then, using arguments analogous to those leading to \eqref{ssn_inclusion} and \eqref{oiFB-cond-general}, one can show that
\begin{equation*}%\label{IFISTA-cond-general}
\Delta^{k,t}:=\frac{L}{\tau}(\widetilde{\bm{x}}^{k,t}-\bm{x}^{k,t})
\in\partial_{\varepsilon_{k,t}}P(\widetilde{\bm{x}}^{k,t})
+ \nabla f(\bm{y}^k) + \frac{L}{\tau}\big(\widetilde{\bm{x}}^{k,t}-\bm{y}^k\big),
\end{equation*}
where $\varepsilon_{k,t}:=r(\widetilde{\bm{x}}^{k,t}_J) - r(\bm{x}^{k,t}_J)
+ \langle A^\top \bm{z}^{k,t} - \nabla f(\bm{y}^k) - \frac{L}{\tau}(\bm{x}^{k,t} - \bm{y}^k), \,\bm{x}^{k,t} - \widetilde{\bm{x}}^{k,t} \rangle$. Thus, the error criterion \eqref{bgk2023on-cond} is verifiable at the triple $(\widetilde{\bm{x}}^{k,t},\Delta^{k,t},\varepsilon_{k,t})$. Specifically, at the $k$th iteration, the subsolver in I-FISTA can be terminated when
\begin{equation}\label{bgk2023on-cond-general}
\|\tau\Delta^{k,t}\|^2 + 2\tau L\varepsilon_{k,t} \leq L\big((1-\tau)L - c\tau\big) \| \widetilde{\bm{x}}^{k,t} - \bm{y}^{k}\|_F^2,
\end{equation}
where $\tau\in(0,1)$ and $c\in[0,\,L(1-\tau)/\tau]$.

In summary, different inexact APG mechanisms lead to different practically verifiable stopping conditions. A key practical distinction is that the proposed SpinAPG allows the algorithm to proceed with a potentially \textit{infeasible} iterate $\bm{x}^{k,t}$ without explicit projections onto $\mathrm{dom}\,P$. Such flexibility is typically absent in existing inexact APG frameworks, where feasibility must be enforced at every iteration.

Finally, we note that the inexact APG method developed in \cite{jst2012inexact} can also alleviate this feasibility issue by approximately solving a perturbed subproblem over an enlarged feasible set. While effective, this framework is relatively complex, tailored to linearly constrained convex semidefinite programs, and does not accommodate general nonsmooth regularizers. In contrast, our shadow-point enhanced error criterion is more natural, flexible, and broadly applicable. Moreover, as shown later, under weaker assumptions, SpinAPG admits stronger theoretical guarantees, including an improved $o(1/k^2)$ convergence rate for the objective function values as well as the iterate convergence, both of which were not addressed in \cite{jst2012inexact}.

%%%%%%%%%%%%%%%%%%%%%%%%%%%%%%%%%%%%%%%%%%%%%%%%
\section{Convergence analysis}\label{sec-convana}

In this section, we study the convergence properties of SpinAPG in Algorithm \ref{algo-SpinAPG}, showing that it preserves all desirable convergence guarantees of the classical APG method under suitable error conditions. Our analysis builds upon insights from existing works on the APG method and its inexact variants (e.g., \cite{ad2015stability,bpr2016variable,brr2018inertial,cd2015convergence,jst2012inexact,srb2011convergence,t2008on}). A key challenge in our analysis arises from the introduction of the shadow-point enhanced error criterion \eqref{inexcond-iAPG}, which provides greater flexibility in handling inexactness at each iteration. While this flexibility enhances the adaptability of the APG method, it also complicates the theoretical analysis by requiring a more refined treatment of the error terms and their impact on the convergence behavior. To address this, we extend existing proof techniques and develop a generalized analytical framework that accommodates three distinct error components: the error term $\Delta^{k}$, the $\varepsilon_k$-subdifferential $\partial_{\varepsilon_k}P$, and the deviation $\|\widetilde{\bm{x}}^{k+1}-\bm{x}^{k+1}\|$. These error terms are carefully controlled through appropriate recursive bounds. Our analysis ultimately demonstrates that SpinAPG remains effective and achieves a convergence rate consistent with that of the classical APG method, provided that these error terms are properly managed.

We now start our convergence analysis by making some blanket technical assumptions and gathering three technical lemmas. The proofs of the last two lemmas are provided in Appendices \ref{proof-lem-thetasumbd} and \ref{proof-lem-recursion}, respectively.

\begin{assumption}\label{assumA}
The functions $P$ and $f$ satisfy the following assumptions.
\begin{itemize}[leftmargin=0.9cm]
\item[{\bf A1.}] $P: \mathbb{E}\rightarrow\mathbb{R}\cup\{+\infty\}$ is a proper closed convex (possibly nonsmooth) function.

\item[{\bf A2.}] $f: \mathbb{E} \rightarrow \mathbb{R}$ is a continuously differentiable convex function with a Lipschitz continuous gradient, i.e., there exists a constant $L > 0$ such that $\|\nabla f(\bm{x}) - \nabla f(\bm{y})\| \leq L \|\bm{x}- \bm{y}\|$ for any $\bm{x}, \,\bm{y} \in \mathbb{E}$.
    %\begin{equation*}
    %\|\nabla f(\bm{x}) - \nabla f(\bm{y})\| \leq L \|\bm{x}- \bm{y}\|, \qquad \forall\,\bm{x}, \,\bm{y} \in \mathbb{E}.
    %\end{equation*}

%\vspace{1mm}
%\item[{\bf A3.}] $F$ is level-bounded, i.e., the level set $\big\{\bm{x}\in\mathbb{E} : F(\bm{x}) \leq \alpha\big\}$ is bounded (possibly empty) for every $\alpha\in\mathbb{R}$.
\end{itemize}
\end{assumption}

\begin{lemma}[{\cite[Section 2.2]{p1987introduction}}]\label{lemseqcon}
Suppose that $\{\alpha_k\}_{k=0}^{\infty}\subseteq\mathbb{R}$ and $\{\gamma_k\}_{k=0}^{\infty}\subseteq\mathbb{R}_+$ are two sequences such that $\{\alpha_k\}$ is bounded from below, $\sum_{k=0}^{\infty} \gamma_k < \infty$, and $\alpha_{k+1} \leq \alpha_{k} + \gamma_k$ holds for all $k\geq0$. Then, $\{\alpha_k\}$ is convergent.
\end{lemma}

\begin{lemma}\label{lem-thetasumbd}
For any sequence $\{\theta_k\}_{k=-1}^{\infty}\subseteq(0,1]$ with $\theta_{-1}=\theta_0=1$ satisfying conditions \eqref{condtheta1} and \eqref{condtheta}, we have that $\theta_k\geq\frac{1}{k+ 2 }$ and $\sum^k_{i=0}\theta_{i-1}^{-2} - \theta_{i}^{-2}(1-\theta_{i}) \leq 2 + \frac{(k+3)^{2}}{2}$ for all $k\geq0$.
% \begin{equation*}
% \sum^k_{i=0}\frac{1}{\theta_{i-1}^{ 2 }} - \frac{1-\theta_{i}}{\theta_{i}^{ 2 }} \leq 2 + \frac{(k+3)^{2}}{2},\quad \mbox{ for all }k\geq0.
% \end{equation*}
\end{lemma}
%\begin{proof}
%The proof can be found in Appendix \ref{proof-lem-thetasumbd}.
%\end{proof}

\begin{lemma}\label{lem-recursion}
Suppose that $\{a_k\}_{k=0}^{\infty}$, $\{q_k\}_{k=0}^{\infty}$, $\{c_k\}_{k=0}^{\infty}$, $\{\lambda_k\}_{k=0}^{\infty}$ and $\{\widetilde{\lambda}_k\}_{k=0}^{\infty}$ are five sequences of nonnegative numbers, and satisfy the following recursion:
\begin{equation*}
a_{k+1}^2 \leq q_k + {\textstyle\sum^k_{i=0}}\big(\lambda_ia_{i+1}+\widetilde{\lambda}_{i}a_i+c_i\big), \quad \forall\,k\geq0.
\end{equation*}
Then, for all $k\geq0$, we have that $a_{k+1} \leq \frac{1}{2} P_k + \sqrt{Q_k + \sum^k_{i=0}c_i + \big(\frac{1}{2}P_k\big)^2}$,
%\begin{equation*}
%a_{k+1}
%\leq {\textstyle\frac{1}{2}\sum^k_{i=0}}\big(\lambda_i+\widetilde{\lambda}_{i-1}\big) + \sqrt{B_k + {\textstyle\sum^k_{i=0}} c_i +
%\big({\textstyle\frac{1}{2}\sum^k_{i=0}}(\lambda_i+\widetilde{\lambda}_{i-1})\big)^2}
%\end{equation*}
and
\begin{equation*}
\begin{aligned}
&\quad {\textstyle\sum^k_{i=0}}\big(\lambda_ia_{i+1}+\widetilde{\lambda}_{i}a_i+c_i\big)
\;\leq P_k^2\;
+ {\textstyle\sum^k_{i=0}}(\lambda_i+\widetilde{\lambda}_{i})\,Q_i^{\frac{1}{2}}
+ P_k
\big({\textstyle\sum^k_{i=0}}c_i\big)^{\frac{1}{2}}
+ {\textstyle\sum^k_{i=0}}c_i,
\end{aligned}
\end{equation*}
where $Q_k:=\max\limits_{0\leq i \leq k}\big\{q_i\big\}$ and $P_k := \sum^k_{i=0}\big(\lambda_i+\widetilde{\lambda}_{i}\big)$.
\end{lemma}
%\begin{proof}
%This technical lemma extends \cite[Lemma 1]{srb2011convergence} to a more general setting, and we provide a simplified proof in Appendix \ref{proof-lem-recursion}.
%\end{proof}

%%%%%%%%%%%%%%%%%%%%%%%%%%%%%%%%%%%%%%%%%%%%%%%%%%
\subsection{The ${\cal O}(1/k^2)$ convergence rate}\label{sec-convana-1}

In this subsection, we show that SpinAPG achieves the expected ${\cal O}(1/k^2)$ convergence rate for the objective function values. For simplicity, we define the following quantities:
\begin{equation}\label{defnots}
\begin{aligned}
&e(\bm{x}):=F(\bm{x})-F^*, ~~\forall\,\bm{x}\in\mathrm{dom}\,P, \\
&\vartheta_k:=\theta_{k-1}^{-2}-\theta_{k}^{-2}(1-\theta_{k}), \quad \overline{\vartheta}_k:={\textstyle\sum^{k}_{i=0}}\,\vartheta_i, \quad \forall\, k\geq 0,\\
&\delta_k:=\theta_k^{-1}\eta_k+2(\alpha-1)L\theta_{k-1}^{-1}\mu_{k-1}, \quad
\overline{\delta}_k:={\textstyle\sum^{k}_{i=0}}\,\delta_i, \quad \forall\, k\geq 0,\\
&\xi_k:=\theta_k^{-2}\big({\textstyle\frac{L}{2}}\mu_k^2+\eta_k(\mu_k+\mu_{k-1})+\nu_k\big), \quad
\overline{\xi}_k:={\textstyle\sum^{k}_{i=0}}\,\xi_i,\quad \forall\, k\geq 0,
\end{aligned}
\end{equation}
where $F^*:=\inf_{\bm{x}\in\mathbb{E}}\left\{F(\bm{x})\right\}$.
Clearly, $e(\bm{x})\geq0$ for any $\bm{x}\in\mathrm{dom}\,P$ and $\vartheta_k\geq0$ for any $k\geq0$. For consistency, we set $\mu_{-1}=1$ and allow $\widetilde{\bm{x}}^{0}\in\mathrm{dom}\,P$ to be arbitrarily chosen.\footnote{The setting of $\mu_{-1}$ and $\widetilde{\bm{x}}^{0}$ is only used for theoretical analysis, and is not needed in practical implementations.} We then establish an approximate sufficient descent property, followed by a key recursive relation.
%which is crucial for establishing the convergence rate of SpinAPG.

\begin{lemma}[\textbf{Approximate sufficient descent property}]\label{lem-appsuffdes}
Suppose that Assumption \ref{assumA} holds.
Let $\{\bm{x}^k\}$, $\{\widetilde{\bm{x}}^k\}$ and $\{\bm{y}^k\}$ be the sequences generated by SpinAPG in Algorithm \ref{algo-SpinAPG}. Then, for any $k\geq0$ and any $\bm{x}\in\mathrm{dom}\,P$,
\begin{equation}\label{suffdes1-iAPG}
F(\widetilde{\bm{x}}^{k+1}) - F(\bm{x})
\leq \frac{L}{2}\|\bm{x}-\bm{y}^{k}\|^2
- \frac{L}{2}\|\bm{x}-\bm{x}^{k+1}\|^2
+ \eta_k\|\widetilde{\bm{x}}^{k+1}-\bm{x}\|
+ \frac{L}{2}\mu_k^2 + \nu_k.
\end{equation}
\end{lemma}
\begin{proof}
See Appendix \ref{proof-lem-appsuffdes}.
\end{proof}

\begin{lemma}\label{lem-suffdes-iAPG}
Suppose that Assumption \ref{assumA} holds. Let $\{\bm{x}^k\}$ and $\{\widetilde{\bm{x}}^k\}$ be the sequences generated by SpinAPG in Algorithm \ref{algo-SpinAPG}, and let $\bm{z}^k := \bm{x}^k + (\theta_{k-1}^{-1}-1)(\bm{x}^k-\bm{x}^{k-1})$ for all $k\geq0$. Then, for any $k\geq 0$ and any $\bm{x}\in\mathrm{dom}\,P$,
\begin{equation}\label{suffdes-iAPG}
\begin{aligned}
&\quad \frac{1-\theta_{k+1}}{\theta_{k+1}^2}
\left(F(\widetilde{\bm{x}}^{k+1})-F(\bm{x})\right) + \frac{L}{2}\|\bm{x} - \bm{z}^{k+1}\|^2  \\[2pt]
&\leq \frac{1-\theta_{k}}{\theta_{k}^2}
\left(F(\widetilde{\bm{x}}^{k})-F(\bm{x})\right)
+ \frac{L}{2}\|\bm{x} - \bm{z}^k\|^2
+ e(\bm{x})\vartheta_{k+1}  \\[3pt]
&\qquad
+ \left(\theta_k^{-1}\eta_k+(\alpha-1)L\theta_{k-1}^{-1}\mu_{k-1}\right)\|\bm{x}-\bm{z}^{k+1}\|
+ (\alpha-1)L\theta_{k-1}^{-1}\mu_{k-1}\|\bm{x}-\bm{z}^k\| + \xi_k,
\end{aligned}
\end{equation}
where $e(\bm{x})$, $\vartheta_k$, $\xi_k$ are defined in \eqref{defnots}, and $\alpha$ is as in \eqref{condtheta1}.
\end{lemma}
\begin{proof}
See Appendix \ref{proof-lem-suffdes-iAPG}.
\end{proof}

We are now ready to establish the convergence rate of SpinAPG.

\begin{theorem}\label{thm1-iAPG}
Suppose that Assumption \ref{assumA} holds. Let $\{\bm{x}^k\}$ and $\{\widetilde{\bm{x}}^k\}$ be the sequences generated by SpinAPG in Algorithm \ref{algo-SpinAPG}. Then, the following results hold. 
\begin{itemize}[leftmargin=0.9cm]
\item[{\rm (i)}] For any $k\geq0$ and any $\bm{x}\in\mathrm{dom}\,P$,
    \begin{equation*}%\label{comp-iAPG}
    \begin{aligned}
    &\quad F(\widetilde{\bm{x}}^{k+1}) - F(\bm{x})  \\
    &\leq \left(\frac{\alpha-1}{k+\alpha-1}\right)^2
    \left(\big({\textstyle\frac{L}{2}}+\overline{\delta}_k\big)\|\bm{x}-\bm{x}^0\|^2
    + e(\bm{x})\overline{\vartheta}_k
    + \sqrt{\frac{2e(\bm{x})}{L}}\,A_k + B_k\right),
    \end{aligned}
    \end{equation*}
    where $A_k:=\overline{\vartheta}_k^{\frac{1}{2}}\overline{\delta}_k + \sum_{i=0}^k\theta_i^{-1}\delta_i$, $B_k:=(2/L)\,\overline{\delta}_k^2+\sqrt{2/L}\,\overline{\delta}_k\,\overline{\xi}_k^{\frac{1}{2}}
    +\overline{\xi}_k$, and $e(\bm{x})$, $\delta_k$, $\overline{\delta}_k$, $\overline{\xi}_k$, $\overline{\vartheta}_k$ are defined in \eqref{defnots}, and $\alpha$ is as in \eqref{condtheta1}.

\item[{\rm (ii)}] If $\frac{1}{k}\sum^{k-1}_{i=0}\theta_i^{-1}\eta_i\to0$, $\frac{1}{k}\sum^{k-1}_{i=0}\theta_i^{-1}\mu_i\to0$ and $\frac{1}{k}\sum^{k-1}_{i=0}\theta_i^{-2}\nu_i\to0$, then $F(\widetilde{\bm{x}}^k) \to F^*$.

\item[{\rm (iii)}] If $\sum\theta_k^{-1}\eta_k<\infty$, $\sum\theta_k^{-1}\mu_k<\infty$, $\sum\theta_k^{-2}\nu_k<\infty$, and the optimal solution set of problem \eqref{mainpro} is nonempty, then 
    \begin{equation*}
    F(\widetilde{\bm{x}}^k) - F^* \leq \mathcal{O}\left(\frac{1}{k^2}\right).
    \end{equation*}
\end{itemize}
\end{theorem}
\begin{proof}
\textit{Statement (i)}.
First, it follows from \eqref{suffdes-iAPG} that, for any $i\geq0$ and any $\bm{x}\in\mathrm{dom}\,P$,
\begin{equation*}
\begin{aligned}
&\quad {\textstyle\theta_{i+1}^{-2}(1-\theta_{i+1})\left(F(\widetilde{\bm{x}}^{i+1})-F(\bm{x})\right) + \frac{L}{2}\|\bm{x}-\bm{z}^{i+1}\|^2}  \\[3pt]
&\leq {\textstyle\theta_{i}^{-2}(1-\theta_{i})\left(F(\widetilde{\bm{x}}^{i})-F(\bm{x})\right)
+ \frac{L}{2}\|\bm{x}-\bm{z}^i\|^2}
+ e(\bm{x})\vartheta_{i+1} \\[3pt]
&\quad
+ \left(\theta_i^{-1}\eta_i+(\alpha-1)L\theta_{i-1}^{-1}\mu_{i-1}\right)\|\bm{x}-\bm{z}^{i+1}\| + (\alpha-1)L\theta_{i-1}^{-1}\mu_{i-1}\|\bm{x}-\bm{z}^i\| + \xi_i,
\end{aligned}
\end{equation*}
where $\bm{z}^i := \bm{x}^i + (\theta_{i-1}^{-1}-1)(\bm{x}^i-\bm{x}^{i-1})$ for all $i\geq0$.
Then, for any $k\geq1$, summing the above inequality from $i=0$ to $i=k-1$ and recalling $\theta_0=1$ yields
\begin{equation*}
\begin{aligned}
&\theta_{k}^{-2}(1-\theta_{k})\big(F(\widetilde{\bm{x}}^{k})-F(\bm{x})\big) + {\textstyle\frac{L}{2}}\|\bm{x} - \bm{z}^{k}\|^2
\leq {\textstyle\frac{L}{2}}\|\bm{x} - \bm{z}^0\|^2
+ e(\bm{x}){\textstyle\sum^{k-1}_{i=0}}\,\vartheta_{i+1} \\
&~~ + {\textstyle\sum^{k-1}_{i=0}}\!
\left(\left(\theta_i^{-1}\eta_i+(\alpha\!-\!1)L\theta_{i-1}^{-1}\mu_{i-1}\right)\|\bm{x}-\bm{z}^{i+1}\|
+ (\alpha\!-\!1)L\theta_{i-1}^{-1}\mu_{i-1}\|\bm{x}-\bm{z}^i\| + \xi_i\right).
\end{aligned}
\end{equation*}
Combining this inequality, \eqref{suffdes-iAPG-tmp}, and $\sum^{k-1}_{i=0}\vartheta_{i+1}=\sum^{k}_{i=1}\vartheta_{i}
\leq\overline{\vartheta}_k:=\sum^{k}_{i=0}\vartheta_i$, one can see that, for all $k\geq1$,
\begin{equation}\label{comp-iAPG-tmp}
\begin{aligned}
\theta_{k}^{-2}\big(F(\widetilde{\bm{x}}^{k+1})-F(\bm{x})\big) + {\textstyle\frac{L}{2}}\|\bm{x}-\bm{z}^{k+1}\|^2
\leq {\textstyle\frac{L}{2}}\|\bm{x}-\bm{z}^0\|^2
+ e(\bm{x})\overline{\vartheta}_k
+ S_k, %\quad \forall\,k\geq1,
\end{aligned}
\end{equation}
where
\begin{equation*}
S_k:={\textstyle\sum^{k}_{i=0}}
\big((\theta_i^{-1}\eta_i+(\alpha-1)L\theta_{i-1}^{-1}\mu_{i-1})\|\bm{x}-\bm{z}^{i+1}\|
+ (\alpha-1)L\theta_{i-1}^{-1}\mu_{i-1}\|\bm{x}-\bm{z}^i\| + \xi_i\big).
\end{equation*}
Moreover, it is easy to verify that \eqref{comp-iAPG-tmp} also holds for $k=0$ with $\mu_{-1}=1$. This together with $F(\bm{x})-F(\widetilde{\bm{x}}^{k+1})\leq e(\bm{x})$ implies that
\begin{equation}\label{xzbound}
{\textstyle\frac{L}{2}}\|\bm{x}-\bm{z}^{k+1}\|^2
\leq {\textstyle\frac{L}{2}}\|\bm{x}-\bm{z}^0\|^2
+ \big(\theta_k^{-2} + \overline{\vartheta}_{k}\big)e(\bm{x})
+ S_k, \quad \forall\,k\geq0.
\end{equation}
Then, by applying Lemma \ref{lem-recursion} with $a_k:=\sqrt{L/2}\,\|\bm{x}-\bm{z}^{k}\|$, $q_k:={\textstyle\frac{L}{2}}\|\bm{x}-\bm{z}^0\|^2
+ \big(\theta_k^{-2} + \overline{\vartheta}_{k}\big)e(\bm{x})$, $c_k:=\xi_k$, $\lambda_k:=\sqrt{2/L}\,(\theta_k^{-1}\eta_k+(\alpha-1)L\theta_{k-1}^{-1}\mu_{k-1})$, $\widetilde{\lambda}_k:= \sqrt{2/L}(\alpha-1)L\,\theta_{k-1}^{-1}\mu_{k-1}$ and noticing that $\{q_k\}$ is nondecreasing in this case because $\theta_{k+1}^{-2}+ \overline{\vartheta}_{k+1}-\theta_k^{-2}-\overline{\vartheta}_{k}=\theta_{k+1}^{-1}>0$ for all $k\geq0$, we obtain that
\begin{equation*}%\label{skbd-iAPG-tmp1}
\begin{aligned}
S_k
&\leq {\textstyle\frac{2}{L}\big({\textstyle\sum^k_{i=0}}\delta_i\big)^2
+ \sqrt{\frac{2}{L}}\sum^{k}_{i=0}\delta_i
\sqrt{\frac{L}{2}\|\bm{x}-\bm{z}^0\|^2
+ \big(\theta_i^{-2}+\overline{\vartheta}_i\big)e(\bm{x})}} \\
&\qquad + {\textstyle\sqrt{\frac{2}{L}}\big(\sum^{k}_{i=0}\delta_i\big)
\big(\sum^k_{i=0}\xi_i\big)^{\frac{1}{2}}
+ \sum^k_{i=0}\xi_i} \\
&\leq {\textstyle \sqrt{\frac{2}{L}}\sum^{k}_{i=0}\delta_i
\sqrt{\frac{L}{2}\|\bm{x}-\bm{z}^0\|^2
+ \big(\theta_i^{-2} + \overline{\vartheta}_i\big)e(\bm{x})}
+ B_k},
\end{aligned}
\end{equation*}
where $\delta_k:=\theta_k^{-1}\eta_k+2(\alpha-1)L\theta_{k-1}^{-1}\mu_{k-1}$, $\overline{\delta}_k:=\sum^{k}_{i=0}\delta_i$, $\overline{\xi}_k:={\textstyle\sum^{k}_{i=0}}\,\xi_i$ and  $B_k:=\frac{2}{L}\overline{\delta}_k^2+\sqrt{\frac{2}{L}}\overline{\delta}_k\overline{\xi}_k^{\frac{1}{2}}
+\overline{\xi}_k$. Moreover,
\begin{equation*}%\label{skbd-iAPG-tmp2}
\begin{aligned}
{\textstyle \sum^{k}_{i=0}\delta_i\sqrt{\frac{L}{2}\|\bm{x}-\bm{z}^0\|^2
+ \big(\theta_i^{-2} + \overline{\vartheta}_i\big)e(\bm{x})}}
&\leq {\textstyle\sum^{k}_{i=0}\delta_i\left(\sqrt{\frac{L}{2}}\|\bm{x}-\bm{z}^0\|
+ \big(\theta_i^{-1}
+ \overline{\vartheta}_i^{\frac{1}{2}}\big)\sqrt{e(\bm{x})}\right)} \\
&\leq {\textstyle \sqrt{\frac{L}{2}}\,\overline{\delta}_k\|\bm{x}-\bm{z}^0\|
+ \sqrt{e(\bm{x})}\Big(\overline{\vartheta}_k^{\frac{1}{2}}\overline{\delta}_k
+ \sum^{k}_{i=0}\theta_i^{-1}\delta_i\Big)}.
%\leq \sqrt{\frac{L}{2}}\,\overline{\delta}_k\|\bm{x}-\bm{z}^0\|
%+ \sqrt{e(\bm{x})}\,A_k,
\end{aligned}
\end{equation*}
where the last inequality follows from the fact that
$0\leq\overline{\vartheta}_0\leq\overline{\vartheta}_1\leq\cdots\leq\overline{\vartheta}_k$. Thus, combining the above two inequalities results in
\begin{equation}\label{skbd-iAPG}
S_k \leq \overline{\delta}_k\|\bm{x}-\bm{z}^0\|
+ \sqrt{2e(\bm{x})/L}\,A_k + B_k,
\end{equation}
where $A_k:=\overline{\vartheta}_k^{\frac{1}{2}} \overline{\delta}_k + \sum_{i=0}^k\theta_i^{-1}\delta_i$. From \eqref{comp-iAPG-tmp} and \eqref{skbd-iAPG}, we further get
\begin{equation*}
\begin{aligned}
\theta_{k}^{-2}\big(F(\widetilde{\bm{x}}^{k+1})-F(\bm{x})\big)
\leq {\textstyle\frac{L}{2}}\|\bm{x}-\bm{z}^0\|^2
+ \overline{\delta}_k\|\bm{x}-\bm{z}^0\|
+ e(\bm{x})\overline{\vartheta}_k
+ \sqrt{2e(\bm{x})/L}\,A_k + B_k.
\end{aligned}
\end{equation*}
Multiplying this inequality by $\theta_{k}^2$, followed by using the fact $\theta_k\leq\frac{\alpha-1}{k+\alpha-1}$ for all $k\geq0$ and recalling $\bm{z}^0=\bm{x}^0$, we can obtain the desired result in statement (i).

\textit{Statement (ii)}.
Notice that, for $k=0$, we have $\theta_{-1}/\theta_0=1<\sqrt{\alpha}$, and for $k\geq1$, we have
\begin{equation*}
\frac{\theta_{k-1}}{\theta_k}
\leq \left(\frac{1}{1-\theta_k}\right)^{\frac{1}{2}}
\leq \left(1+\frac{\alpha-1}{k}\right)^{\frac{1}{2}}
\leq \sqrt{\alpha},
\end{equation*}
where the first inequality follows from \eqref{condtheta} and the second  follows from \eqref{condtheta1}. Then, we see
\begin{equation*}
\begin{aligned}
%&\theta_k^{-1}\delta_k = \theta_k^{-2}\eta_k + 2(\alpha-1)L\theta_k^{-1}\theta_{k-1}^{-1}\mu_{k-1}
%\leq \theta_k^{-2}\eta_k + 2\sqrt{\alpha}(\alpha-1)L\theta_{k-1}^{-2}\mu_{k-1}, %\\
\xi_k &=\theta_k^{-2}\left({\textstyle\frac{L}{2}}\mu_k^2+\eta_k(\mu_k+\mu_{k-1})+\nu_k\right) \\
&\leq {\textstyle\frac{L}{2}}\left(\theta_k^{-1}\mu_k\right)^2
+ \left(\theta_k^{-1}\eta_k\right)
\left(\theta_k^{-1}\mu_k+\sqrt{\alpha}\,\theta_{k-1}^{-1}\mu_{k-1}\right)
+ \theta_k^{-2}\nu_k,
\end{aligned}
\end{equation*}
which implies that
\begin{equation}\label{xiallbd}
\begin{aligned}
\overline{\xi}_k
&%= {\textstyle\sum^{k}_{i=0}}\,\xi_i
\leq {\textstyle\frac{L}{2}}{\textstyle\sum^k_{i=0}}\left(\theta_i^{-1}\mu_i\right)^2
+ {\textstyle\sum^k_{i=0}}\left(\theta_i^{-1}\eta_i\right)
\left(\theta_i^{-1}\mu_i
+ \sqrt{\alpha}\,\theta_{i-1}^{-1}\mu_{i-1}\right)
+ {\textstyle\sum^k_{i=0}}\theta_i^{-2}\nu_i \\
&\leq {\textstyle\frac{L}{2}}\!\left({\textstyle\sum^k_{i=0}}\theta_i^{-1}\mu_i\right)^2
\!\!+\! {\textstyle\frac{1}{2}}\!\left({\textstyle\sum^k_{i=0}}\theta_i^{-1}\eta_i\right)^2
\!+\! {\textstyle\frac{1}{2}}\!\left({\textstyle\sum^k_{i=0}}\!\left(\theta_i^{-1}\mu_i
\!+\! \sqrt{\alpha}\,\theta_{i-1}^{-1}\mu_{i-1}\right)\right)^2 
+ {\textstyle\sum^k_{i=0}}\theta_i^{-2}\nu_i.
\end{aligned}
\end{equation}
Using this relation, $\frac{1}{k}\sum^{k-1}_{i=0}\theta_i^{-1}\eta_i\to0$, $\frac{1}{k}\sum^{k-1}_{i=0}\theta_i^{-1}\mu_i\to0$ and $\frac{1}{k}\sum^{k-1}_{i=0}\theta_i^{-2}\nu_i\to0$, one can verify
\begin{equation}\label{limfacts1}
\lim_{k\to\infty}k^{-1}\overline{\delta}_k = 0, \quad
\lim_{k\to\infty}k^{-2}\overline{\xi}_k = 0, \quad
\lim_{k\to\infty}k^{-2}B_k = 0.
%\begin{aligned}
%&\lim_{k\to\infty}k^{-1}\overline{\delta}_k
%=\lim_{k\to\infty}k^{-1}{\textstyle\sum^k_{i=0}}\left(\theta_i^{-1}\eta_i
%+2(\alpha-1)L\theta_{i-1}^{-1}\mu_{i-1}\right) = 0, \\
%&\lim_{k\to\infty}k^{-2}\overline{\xi}_k=\lim_{k\to\infty}, \\
%&\lim_{k\to\infty}B_k<\infty
%\end{aligned}
\end{equation}
On the other hand, by applying Lemma \ref{lem-thetasumbd}, we have that $\overline{\vartheta}_k \leq 2+\frac{(k+3)^2}{2}$. Thus,
\begin{equation*}
\left(\frac{\alpha-1}{k+\alpha-1}\right)^2\overline{\vartheta}_k \leq 2(\alpha-1)^2, \qquad \left(\frac{\alpha-1}{k+\alpha-1}\right)^2\overline{\vartheta}_k^{\frac{1}{2}}\overline{\delta}_k
\leq \frac{\sqrt{2}(\alpha-1)^2}{k+\alpha-1}\overline{\delta}_k \to 0.
\end{equation*}
Moreover, we have
\begin{equation*}
\left(\frac{\alpha-1}{k+\alpha-1}\right)^2\sum_{i=0}^k\theta_i^{-1}\delta_i
\leq \left(\frac{\alpha-1}{k+\alpha-1}\right)^2\sum_{i=0}^k(i+2)\delta_i
\leq \frac{(\alpha-1)^2}{k+\alpha-1}\,\overline{\delta}_k \to 0,
\end{equation*}
where the first inequality follows from $\theta_{k}\geq\frac{1}{k+2}$ for all $k\geq0$ (by Lemma \ref{lem-thetasumbd}) and the last inequality follows from $\frac{i+2}{k+\alpha-1}\leq1$ for all $0\leq i \leq k$ and $\alpha\geq3$. Then, we see that
\begin{equation*}
\lim_{k\to\infty}\left(\frac{\alpha-1}{k+\alpha-1}\right)^2A_k
= \lim_{k\to\infty}\left(\frac{\alpha-1}{k+\alpha-1}\right)^2
\left(\overline{\vartheta}_k^{\frac{1}{2}} \overline{\delta}_k + \sum_{i=0}^k\theta_i^{-1}\delta_i\right)
= 0.
\end{equation*}
Now, using this fact, \eqref{limfacts1} and statement (i), we can conclude that
\begin{equation*}
\limsup\limits_{k\to\infty}\,F(\widetilde{\bm{x}}^{k})
\leq F(\bm{x}) + 2(\alpha-1)^2\,e(\bm{x}), \quad \forall\,\bm{x}\in\mathrm{dom}\,P.
\end{equation*}
This implies that
\begin{equation*}
F^*
\leq\liminf\limits_{k\to\infty}\,F(\widetilde{\bm{x}}^{k})
\leq\limsup\limits_{k\to\infty}\,F(\widetilde{\bm{x}}^{k})
\leq F^*,
\end{equation*}
from which, we can conclude that $F(\widetilde{\bm{x}}^{k}) \to F^*$ and proves statement (ii).

\textit{Statement (iii)}.
If, in addition, the optimal solution set of problem \eqref{mainpro} is nonempty, we have that $F^*=\min\left\{F(\bm{x}) : \bm{x}\in\mathbb{E} \right\} = F(\bm{x}^*)$, where $\bm{x}^*$ is an optimal solution of problem \eqref{mainpro}. Then, applying statement (i) to $\bm{x}=\bm{x}^*$, we get
\begin{equation*}
0 \leq F(\widetilde{\bm{x}}^{k+1}) - F^* \leq \left(\frac{\alpha-1}{k+\alpha-1}\right)^2
\left(\big(L/2+\overline{\delta}_k\big)\|\bm{x}^*-\bm{x}^0\|^2 + B_k\right), \quad \forall\,k\geq0.
\end{equation*}
Moreover, from $\sum\theta_k^{-1}\eta_k<\infty$, $\sum\theta_k^{-1}\mu_k<\infty$, $\sum\theta_k^{-2}\nu_k<\infty$ and \eqref{xiallbd}, one can verify that
\begin{equation*}
\lim_{k\to\infty}\overline{\delta}_k<\infty, \quad
\lim_{k\to\infty}\overline{\xi}_k<\infty, \quad
\lim_{k\to\infty}B_k<\infty.
\end{equation*}
By the above facts, we can conclude that $F(\widetilde{\bm{x}}^k) - F^* \leq \mathcal{O}(1/k^2)$ and complete the proof.
\end{proof}

In Theorem \ref{thm1-iAPG}, we establish the ${\cal O}(1/k^2)$ convergence rate of SpinAPG in terms of the objective residual $F(\widetilde{\bm{x}}^k) - F^*$. Notably, all these results are derived under standard assumptions commonly used in the analysis of the APG method, along with suitable summable error conditions on $\{\eta_k\}_{k=0}^{\infty}$, $\{\mu_k\}_{k=0}^{\infty}$ and $\{\nu_k\}_{k=0}^{\infty}$. In particular, our analysis does not require additional stringent conditions, such as the boundedness of the dual solution sequence assumed in \cite{jst2012inexact}. This highlights an advantage of our inexact framework over the one proposed in \cite{jst2012inexact}.

%%%%%%%%%%%%%%%%%%%%%%%%%%%%%%%%%%%%%%%%%%%%%%%
\subsection{The improved $o(1/k^2)$ convergence rate and the sequential convergence}\label{sec-convana-2}

In this subsection, we further explore the improved convergence properties of SpinAPG under specific choices of the parameter sequence $\{\theta_k\}$. Motivated by recent works \cite{acpr2018fast,ap2016rate,brr2018inertial,cd2015convergence}, we establish that, if $\theta_{k}$ in \textbf{Step 3} of Algorithm \ref{algo-SpinAPG} is set as
\begin{equation}\label{improvetheta}
\theta_k=\frac{\alpha-1}{k+\alpha-1} \quad \mbox{with} \quad  \alpha>3,
\end{equation}
for all $k\geq1$, the convergence rate of the function value sequence generated by SpinAPG improves to $o(1/k^2)$, and both sequences $\{\bm{x}^k\}$ and $\{\widetilde{\bm{x}}^k\}$ are guaranteed to converge to the same limit.
Our results extend the relevant conclusions from \cite{acpr2018fast,ap2016rate,brr2018inertial,cd2015convergence}, making them applicable to a broader and more practical inexact framework. We first give a lemma that is crucial for proving the improved convergence rate.

\begin{lemma}\label{lem-smallo}
Suppose that Assumption \ref{assumA} holds, and the optimal solution set of problem \eqref{mainpro} is nonempty. Additionally, assume that $\theta_k$ is chosen as \eqref{improvetheta}, and $\sum\theta_k^{-1}\eta_k<\infty$, $\sum\theta_k^{-1}\mu_k<\infty$, $\sum\theta_k^{-2}\nu_k<\infty$. Let $\{\bm{x}^k\}$ and $\{\widetilde{\bm{x}}^k\}$ be the sequences generated by SpinAPG in Algorithm \ref{algo-SpinAPG}. Then, the following statements hold.
\begin{itemize}[leftmargin=0.9cm]
\item[{\rm (i)}] $\sum^{\infty}_{k=0}k\big(F(\widetilde{\bm{x}}^{k})-F^*\big) < \infty$.
\item[{\rm (ii)}] $\sum^{\infty}_{k=0}k\|\bm{x}^{k}-\bm{x}^{k-1}\|^2 < \infty$.
\item[{\rm (iii)}] $\lim\limits_{k\to\infty}k^2\big(F(\widetilde{\bm{x}}^k)-F^*\big) + \frac{L}{2}k^2\|\bm{x}^{k}-\bm{x}^{k-1}\|^2$ exists.
\end{itemize}
\end{lemma}
\begin{proof}
See Appendix \ref{proof-lem-smallo}.
\end{proof}

With this lemma in hand, we now present the main results in the following theorem.
%the improved convergence results for the objective function value, as well as the sequential convergence, in the following theorem.

\begin{theorem}
%Suppose that Assumption \ref{assumA} holds with \red{$\mathcal{Q}=\mathcal{X}=\mathbb{E}$ and the kernel function $\phi(\bm{x})=\frac{1}{2}\|\bm{x}\|^2$. Also, suppose that $\theta_k$ is chosen by \eqref{improvetheta}, $\sum\theta_k^{-1}\eta_k<\infty$, $\sum\theta_k^{-1}\mu_k<\infty$, $\sum\theta_k^{-2}\nu_k<\infty$ and the optimal solution set of problem \eqref{mainpro} is nonempty.} Let $\{\bm{x}^k\}$ and $\{\widetilde{\bm{x}}^k\}$ be the sequences generated by the iAPG in Algorithm \ref{algo-SpinAPG}.
Under the same conditions of Lemma \ref{lem-smallo}, the following statements hold.
\begin{itemize}[leftmargin=0.9cm]
\item[{\rm (i)}] $\lim\limits_{k\to\infty} k^2\big(F(\widetilde{\bm{x}}^{k})-F^*\big)=0$ and $\lim\limits_{k\to\infty} k\|\bm{x}^{k}-\bm{x}^{k-1}\|=0$. In other words,
    \begin{equation*}
    F(\widetilde{\bm{x}}^{k})-F^* = o\left(\frac{1}{k^{2}}\right)
    \quad \text{and} \quad
    \|\bm{x}^{k}-\bm{x}^{k-1}\| = o\left(\frac{1}{k}\right).
    \end{equation*}

\item[{\rm (ii)}] Both sequences $\{\bm{x}^k\}$ and $\{\widetilde{\bm{x}}^k\}$ are bounded, and any cluster point of $\{\bm{x}^k\}$ or $\{\widetilde{\bm{x}}^k\}$ is an optimal solution of problem \eqref{mainpro}.

\item[{\rm (iii)}] $\{\bm{x}^{k}\}$ and $\{\widetilde{\bm{x}}^{k}\}$ converge to a same optimal solution of problem \eqref{mainpro}.
\end{itemize}
\end{theorem}
\begin{proof}
\textit{Statement (i)}.
It readily follows from Lemma \ref{lem-smallo}(i)\&(ii) that
\begin{equation*}
\sum\frac{1}{k}\big( k^2\big(F(\widetilde{\bm{x}}^{k})-F(\bm{x}^*)\big) + {\textstyle\frac{L}{2}}k^2\|\bm{x}^{k}-\bm{x}^{k-1}\|^2 \big) < \infty.
\end{equation*}
This, together with Lemma \ref{lem-smallo}(iii), implies that
\begin{equation*}
\lim\limits_{k\to\infty}k^2\big(F(\widetilde{\bm{x}}^k)-F(\bm{x}^*)\big)
+ {\textstyle\frac{L}{2}}k^2\|\bm{x}^{k}-\bm{x}^{k-1}\|^2=0.
\end{equation*}
Since two terms are nonnegative, their limits are equal to 0. This proves statement (i).

\vspace{1mm}
\textit{Statement (ii)}.
Recall that $\bm{z}^k:=\bm{x}^k + (\theta_{k-1}^{-1}-1)(\bm{x}^k-\bm{x}^{k-1})=\bm{x}^k + \frac{k-1}{\alpha-1}(\bm{x}^k-\bm{x}^{k-1})$. Given an optimal solution of \eqref{mainpro}, denoted by $\bm{x}^*$, one can see from \eqref{xzbound}, \eqref{skbd-iAPG} and $e(\bm{x}^*)=0$ that
\begin{equation*}
{\textstyle\frac{L}{2}}\|\bm{x}^*-\bm{z}^{k+1}\|^2
\leq \big({\textstyle\frac{L}{2}}+\overline{\delta}_k\big)\|\bm{x}^*-\bm{z}^0\|^2
+ B_k, \quad \forall\,k\geq0,
\end{equation*}
where $B_k:=(2/L)\,\overline{\delta}_k^2+\sqrt{2/L}\,\overline{\delta}_k\,\overline{\xi}_k^{\frac{1}{2}}
+\overline{\xi}_k$ with $\overline{\delta}_k$ and $\overline{\xi}_k$ defined in \eqref{defnots}. Since $\sum\theta_k^{-1}\eta_k<\infty$, $\sum\theta_k^{-1}\mu_k<\infty$, $\sum\theta_k^{-2}\nu_k<\infty$ and $\theta_k\in(0,1]$, we have $\lim\limits_{k\to\infty}\overline{\delta}_k<\infty$ and $\lim\limits_{k\to\infty}B_k<\infty$. Thus, $\{\|\bm{x}^*-\bm{z}^{k}\|\}$ is bounded and hence $\{\bm{z}^k\}$ is bounded. Moreover, we know from statement (i) that $\{k\|\bm{x}^{k}-\bm{x}^{k-1}\|\}$ (and hence $\{k(\bm{x}^{k}-\bm{x}^{k-1})\}$) is also bounded. Therefore, $\{\bm{x}^k\}$ is bounded. This then together with $\|\widetilde{\bm{x}}^{k}-\bm{x}^{k}\|\leq\mu_{k-1}$ and $\sum\mu_k<\infty$ implies that $\{\widetilde{\bm{x}}^{k}\}$ is also bounded.

Since $\{\widetilde{\bm{x}}^{k}\}$ is bounded, it has at least one cluster point. Suppose that $\widetilde{\bm{x}}^{\infty}$ is a cluster point and $\{\widetilde{\bm{x}}^{k_i}\}$ is a convergent subsequence such that $\lim_{i\to\infty} \widetilde{\bm{x}}^{k_i} = \widetilde{\bm{x}}^{\infty}$. Then, from statement (i) and the lower semicontinuity of $F$ (since $P$ and $f$ are closed by Assumptions \ref{assumA}2\&3), we see that $F^* = \lim_{k\to\infty}\,F(\widetilde{\bm{x}}^{k}) \geq F(\widetilde{\bm{x}}^{\infty})$. This implies that $F(\widetilde{\bm{x}}^{\infty})$ is finite and hence  $\widetilde{\bm{x}}^{\infty}\in\mathrm{dom}\,F$. Therefore, $\widetilde{\bm{x}}^{\infty}$ is an optimal solution of problem \eqref{mainpro}.

Note also that $\{\bm{x}^{k}\}$ is bounded and has at least one cluster point. Suppose that $\bm{x}^{\infty}$ is a cluster point and $\{\bm{x}^{k_i}\}$ is a convergent subsequence such that $\lim_{i\to\infty}\bm{x}^{k_i}=\bm{x}^{\infty}$. This then together with $\|\widetilde{\bm{x}}^{k_i}-\bm{x}^{k_i}\|\leq\mu_{k_i-1}\to0$ induces that $\{\widetilde{\bm{x}}^{k_i}\}$ also converges to $\bm{x}^{\infty}$. Hence, $\bm{x}^{\infty}$ is a cluster point of $\{\widetilde{\bm{x}}^{k}\}$ and thus, from discussions in last paragraph, $\bm{x}^{\infty}$ is also an optimal solution of \eqref{mainpro}. This proves statement (ii).

\vspace{1mm}
\textit{Statement (iii)}.
Let $\bm{x}^*$ be an arbitrary optimal solution of problem \eqref{mainpro}. It follows from \eqref{suffdes-iAPG}, \eqref{skbd-iAPG}, $e(\bm{x}^*)=0$, $\lim_{k\to\infty}\overline{\delta}_k<\infty$, $\lim_{k\to\infty}B_k<\infty$ and Lemma \ref{lemseqcon} that $\big\{(1-\theta_{k})\theta_{k}^{-2}\big(F(\widetilde{\bm{x}}^{k})-F(\bm{x}^*)\big) +(L/2)\|\bm{x}^*-\bm{z}^k\|^2\big\}$ is convergent. Note that $(1-\theta_{k})\theta_{k}^{-2}=\frac{k(k+\alpha-1)}{(\alpha-1)}$, $\lim_{k\to\infty} k^2\big(F(\widetilde{\bm{x}}^{k})-F^*\big)=0$ (by statement (i)) and $\lim_{k\to\infty} k\big(F(\widetilde{\bm{x}}^{k})-F^*\big)=0$ (by Lemma \ref{lem-smallo}(i)). Then, $\lim_{k\to\infty}\|\bm{z}^k-\bm{x}^*\|$ exists. Moreover, observe that
\begin{equation*}
\begin{aligned}
\|\bm{z}^k-\bm{x}^*\|^2
&= \|\bm{x}^k
+ {\textstyle\frac{k-1}{\alpha-1}}(\bm{x}^k-\bm{x}^{k-1})-\bm{x}^*\|^2 \\
&= {\textstyle\left(\frac{k-1}{\alpha-1}\right)^2}\|\bm{x}^k-\bm{x}^{k-1}\|^2
+ 2{\textstyle\left(\frac{k-1}{\alpha-1}\right)}
\langle\bm{x}^k-\bm{x}^{k-1},\,\bm{x}^k-\bm{x}^*\rangle
+ \|\bm{x}^k-\bm{x}^*\|^2.
\end{aligned}
\end{equation*}
From this, $\lim_{k\to\infty}k\|\bm{x}^{k}-\bm{x}^{k-1}\|=0$ (by statement (i)), the boundedness of $\{\bm{x}^k\}$ (by statement (ii)) and the fact that $\lim_{k\to\infty}\|\bm{z}^k-\bm{x}^*\|$ exists, we see that $\lim_{k\to\infty}\|\bm{x}^k-\bm{x}^*\|$ exists.

Suppose that $\bm{x}^{\infty}$ is a cluster point of $\{\bm{x}^k\}$ (which must exist since $\{\bm{x}^k\}$ is bounded) and $\{\bm{x}^{k_j}\}$ is a convergent subsequence such that $\lim_{j\to\infty} \bm{x}^{k_j} = \bm{x}^{\infty}$. Note from statement (ii) that $\bm{x}^{\infty}$ is an optimal solution of problem \eqref{mainpro}. Then, from discussions in the last paragraph, we can conclude that $\lim_{k\to\infty}\|\bm{x}^k-\bm{x}^{\infty}\|$ exists. This together with $\lim_{j\to\infty} \bm{x}^{k_j} = \bm{x}^{\infty}$ implies that $\lim_{k\to\infty}\|\bm{x}^k-\bm{x}^{\infty}\|=0$. Now, let $\bm{z}$ be an arbitrary cluster point of $\{\bm{x}^k\}$ with a convergent subsequence $\{\bm{x}^{k'_j}\}$ such that $\bm{x}^{k'_j}\to\bm{z}$. Since $\lim_{k\to\infty}\|\bm{x}^k-\bm{x}^{\infty}\|=0$, we have that $\lim_{k\to\infty}\|\bm{x}^{k'_j}-\bm{x}^{\infty}\|=0$. This further implies that $\bm{z}=\bm{x}^{\infty}$. Since $\bm{z}$ is arbitrary, we can conclude that $\lim_{k\to\infty}\bm{x}^k=\bm{x}^{\infty}$. Finally, using this together with $\|\widetilde{\bm{x}}^{k}-\bm{x}^{k}\|\leq\mu_{k-1}\to0$, we deduce that $\{\widetilde{\bm{x}}^k\}$ also converges to $\bm{x}^{\infty}$. This completes the proof.
\end{proof}

%%%%%%%%%%%%%%%%%%%%%%%%%%%%%%%%%%%%%%%%%%%%%
\section{Numerical experiments}\label{sec-num}

In this section, we conduct numerical experiments to compare SpinAPG with several representative inexact APG schemes: iAPG-SLB \cite{srb2011convergence}, AIFB \cite{vsbv2013accelerated}, o-iFB \cite{ad2015stability}, and I-FISTA \cite{bgk2023on}. %Our goal \magenta{[LMX: the following sentence is a little bit strange]} is not to develop a state-of-the-art method for solving a specific problem, but rather to examine and contrast the practical behavior of different inexact APG mechanisms. 
The results provide insights that complement the theoretical analysis and underscore the importance of developing a more flexible and economical inexact APG framework.
%All experiments are run in {\sc Matlab} R2023a on a PC with Intel processor i7-12700K@3.60GHz (with 12 cores and 20 threads) and 64GB of RAM, equipped with a Windows OS.
All experiments are conducted in {\sc Matlab} R2024b on an Apple M3 system running macOS %(version 15.3.1) 
with 24 GB of RAM.

%%%%%%%%%%%%%%%%%%%%%%%%%%%%%%%%%%%%%%%%%%%%%%%%%%%%%%%%%%%%%%%%
%\subsection{Sparse quadratic programming}

We consider the following sparse quadratic programming (QP) problem:
\begin{equation}\label{eq:QP_org} 
\begin{aligned} 
\min_{\bm{u}\in\mathbb{R}^n}~\frac{1}{2}\bm{u}^\top P_0\bm{u} + \bm{q}_0^\top\bm{u} +
\lambda \|\bm{u}\|_1 \qquad {\rm s.t.} \qquad A_0\bm{u}\leq\bm{b}, \end{aligned} \end{equation}
where $P_0\in\mathbb{S}_{++}^n$ (symmetric positive definite), $\bm{q}_0\in\mathbb{R}^n$, $A_0\in\mathbb{R}^{m\times n}$, $\bm{b}\in\mathbb{R}^m$, and $\lambda\geq0$. By introducing the slack variable $\bm{s}:=\bm{b}-A_0\bm{u} \in \mathbb{R}^m$ and defining the augmented variable $\bm{x} := (\bm{u}; \bm{s}) \in \mathbb{R}^{n+m}$, problem \eqref{eq:QP_org} can be reformulated as
\begin{equation}\label{eq:QP_orgre}
\begin{aligned}
\min_{\bm{x}\in \mathbb{R}^{n+m}} \ \frac{1}{2} \bm{x}^\top P\bm{x}
+ \bm{q}^\top\bm{x} + \lambda\|\bm{x}_{J}\|_1
+ \delta_{\{\bm{x}_{J^\complement}\ge 0\}}(\bm{x})
\qquad {\rm s.t.} \qquad A\bm{x} = \bm{b},
\end{aligned}
\end{equation}
where $J:=\{1,\dots,n\}$ and $J^\complement := \{1,\dots,n+m\}\!\setminus\!J$ are index sets, and the augmented problem data are given by 
$P = \begin{bmatrix}
P_0 & 0_{n\times m}\\
0_{m\times n} & 0_{m\times m}
\end{bmatrix}$, $\bm{q} = \begin{bmatrix} \bm{q}_0 \\
\bm 0_{m}\end{bmatrix}$, and $A = \begin{bmatrix}
A_0 & I_m
\end{bmatrix}$.

When the feasible set of \eqref{eq:QP_org} is nonempty, standard results such as \cite[Theorem 3.34]{r2006nonlinear}, indicate that $\bm{x}^*$ is an optimal solution of \eqref{eq:QP_orgre} if and only if there exists $\bm{z}^*\in\mathbb{R}^m$ such that
\begin{align*}
A\bm{x} = \bm{b},\quad 0\in P\bm{x}^* + \bm{q} - A^\top \bm{z}^* + \partial g(\bm{x}^*),
\end{align*}
where $g(\bm{x}) :=  \lambda\|\bm{x}_{J}\|_1
+ \delta_{\{\bm{x}_{J^\complement}\ge 0\}}(\bm{x})$. 
Based on the above KKT system, we define the relative KKT residual for any $(\bm{x},\,\bm{z})\in\mathbb{R}^{m+n}\times\mathbb{R}^m$ as:
\begin{equation*}
\Delta_{\rm kkt}(\bm{x},\,\bm{z}) := \max\left\{
\frac{\|A\bm{x}-\bm{b}\|}{1+ \|\bm{b}\|},
\,\frac{\|\bm{x} - \texttt{prox}_{g/L}\big(\bm{x}- P \bm{x}/L -\bm{q}/L+ A^{\top}\bm{z} /L\big) \|}{1+\|\bm{x} \|}
\right\},
\end{equation*}
where $L = \|P\| = \|P_0\|$.
In the subsequent experiments, we will use $\Delta_{\rm kkt}$ to evaluate the accuracy of an approximate solution to problem \eqref{eq:QP_orgre}.

%%%%%%%%%%%%%%%%%%%%%%%%%%%%%%%%%%%%%%%%%%%%%%%%%
\subsection{Implementation of the inexact APG method}\label{sec:implement_APG}

We apply the inexact APG method to the reformulation \eqref{eq:QP_orgre} by setting $f(\bm{x}) := \frac{1}{2}\bm{x}^\top P\bm{x} + \bm{q}^\top\bm{x}$ and $P(\bm{x}) := g(\bm{x}) + \delta_{\{A\bm{x}=\bm{b}\}}(\bm{x})$. Clearly, $\nabla f(\bm{x})=P\bm{x}+\bm{q}$ is $L$-Lipschitz. For the extrapolated point $\bm{y}^k$, we simply adopt the update scheme used in FISTA for all inexact APG methods, except I-FISTA. Specifically, at the $k$th iteration, we compute $\bm{y}^k = \bm{x}^k + \theta_k(\theta_{k-1}^{-1}-1)(\bm{x}^k-\bm{x}^{k-1})$ with $\theta_{k} = \frac{1}{2}\left(\sqrt{\theta_{k-1}^4+4\theta_{k-1}^2}-\theta_{k-1}^2\right)$ for all $k\geq1$. Under this setting, iAPG-SLB and AIFB are equivalent, as they share the same error verification procedure (see Section \ref{sec-example}). To guarantee convergence, I-FISTA should update $\bm{y}^{k}$ as $\bm{y}^{k} = \bm{x}^{k} + \theta_k(\theta_{k-1}^{-1}-1) (\bm{x}^k - \bm{x}^{k-1}) - \tau L^{-1}\theta_k\theta_{k-1}^{-1}\Delta^{k-1}$, where $\Delta^{k-1}$ is the associated error vector satisfying \eqref{bgk2023on-cond-general} from the previous iteration.

At the $k$th iteration, the APG subproblem is
\begin{equation*}%\label{eq:apg_sub_rqp}
\min\limits_{\bm{x}\in\mathbb{R}^{n+m}}~
\left\{ g(\bm{x})
+ \langle \nabla f(\bm{y}^k), \,\bm{x}-\bm{y}^k\rangle
+ \frac{L}{2}\|\bm{x}-\bm{y}^k\|^2 \ \middle\vert\  A\bm{x}=\bm{b} \right\}.
\end{equation*}
Using the same arguments as in Section \ref{sec-example}, this subproblem can be solved on the dual side via a semismooth Newton ({\sc Ssn}) method; see \cite[Section~3]{lst2020on} for more details. Moreover, at each outer iteration, the {\sc Ssn} method is terminated based on the specific requirements of each algorithm.

Specifically, for \textbf{SpinAPG} and \textbf{o-iFB}, the {\sc Ssn} method is terminated when
\begin{equation}\label{stopcond-qp-SpinAPG}
\max\left\{\|A^\top \bm{z}^{k,t} - \nabla f(\bm{y}^k) - L(\bm{x}^{k,t} -\bm{y}^k)\|,\,1\right\}
\cdot\|\nabla\Psi_k(\bm{z}^{k,t})\|
\leq \max\left\{\frac{\Upsilon}{(k+1)^p}, \,10^{-10}\right\},
\end{equation}
where $\Psi_k$ denotes the dual objective at the $k$th iteration defined in \eqref{eq:def_dual}, the coefficient $\Upsilon$ controls the initial accuracy for solving the subproblem, and $p$ determines the tolerance decay rate. Note that \textbf{o-iFB} further requires an explicit feasible $\widetilde{\bm{x}}^{k,t}$ for its next update. In contrast, for \textbf{iAPG-SLB} and \textbf{AIFB}, the {\sc Ssn} method is run until
\begin{equation}\label{stopcond-qp-slb}
g(\widetilde{\bm{x}}^{k,t})
+ \langle \nabla f(\bm{y}^k), \,\widetilde{\bm{x}}^{k,t}-\bm{y}^k\rangle
+ \frac{L}{2}\|\widetilde{\bm{x}}^{k,t}-\bm{y}^k\|^2 + \Psi_k(\bm{z}^{k,t})
\leq \max\left\{\frac{\Upsilon}{(k+1)^p}, \,10^{-10}\right\}.
\end{equation}
Finally, for \textbf{I-FISTA}, the stopping condition of the {\sc Ssn} method is 
\begin{equation}\label{stopcond-qp-ifista}
\|L(\widetilde{\bm{x}}^{k,t}-\bm{x}^{k,t})\|^2 + 2\tau L\varepsilon_{k,t} \leq L\big((1-\tau)L - c\tau\big) \| \widetilde{\bm{x}}^{k,t} - \bm{y}^{k}\|_F^2,
\end{equation}
where $\varepsilon_{k,t}:=\lambda\|\widetilde{\bm{x}}^{k,t}_J\|_1 - \lambda\|\bm{x}^{k,t}_J\|_1 + \langle A^\top \bm{z}^{k,t} - \nabla f(\bm{y}^k) - \frac{L}{\tau}(\bm{x}^{k,t} - \bm{y}^k), \,\bm{x}^{k,t} - \widetilde{\bm{x}}^{k,t} \rangle$, $\tau\in(0,1)$, and $c\in[0,\,L(1-\tau)/\tau]$.

We emphasize again that among all methods above, only SpinAPG allows a potentially \textit{infeasible} iterate $\bm{x}^{k,t}$. In contrast, the other four methods require the explicit computation of a \textit{feasible} point $\widetilde{\bm{x}}^{k,t}$ via projection onto $\mathrm{dom}\,P$. To mitigate the computational burden associated with this projection step, we compute the projection using the state-of-the-art {\sc Ssn} method, which consistently outperforms \texttt{Gurobi} and {\sc Matlab}'s built-in \texttt{quadprog} in our tests, with the relative KKT residual set to $10^{-12}$.

We initialize all methods with $\bm{x}^0:=\bm{0}_{n+m}$, and terminate them when 
\begin{equation*}
\Delta_{\rm kkt}(\bm{x}^k,\bm{z}^k) < 10^{-6},
\end{equation*}
where $\bm{x}^{k}$ and $\bm{z}^{k}$ are approximate primal and dual solutions, respectively, recovered from the {\sc Ssn} iterates as discussed in Section \ref{sec-example}.
% Moreover, we also terminate all methods when the number of total {\sc Ssncg} iterations exceeds $10000$.
To improve efficiency, we employ a \textit{warm-start} strategy to initialize the subsolver {\sc Ssn}. Specifically, at each iteration, we initialize {\sc Ssn} with $\bm{z}^{k,0}=\bm{z}^{k}$, where $\bm{z}^{k}$ is the iterate obtained by {\sc Ssn} in the previous outer iteration.

%%%%%%%%%%%%%%%%%%%%%%%%%%%%%%%%%%%%%%%%%%%%%%%%%%%%%%%
\subsection{Comparison results}\label{sec:qp_comp}

We use the procedure in \cite[Section A.1]{sbgbb2020osqp} to generate problem instances. Specifically, the positive definite matrix $P_0 \in \mathbb{S}^n_{++}$ is constructed as $P_0 = MM^\top + 0.01 I_n$, where $M \in \mathbb{R}^{n \times n}$ is a sparse matrix with 15\% nonzero entries sampled i.i.d. from $\mathcal{N}(0,1)$. The constraint matrix $A_0 \in \mathbb{R}^{m \times n}$ is generated using the same procedure as for $M$, where $m=10n$. To ensure that the feasible set of \eqref{eq:QP_org} is nonempty, we set $\bm{u} = A_0\bm{v} + \bm{\delta}$, where the reference vector $\bm{v} \in \mathbb{R}^n$ has 15\% nonzero entries sampled from $\mathcal{N}(0,1)$, and the noisy vector $\bm{\delta} \in \mathbb{R}^m$ is sampled with entries uniformly distributed in $[0,1]$. The vector $\bm{q}_0 \in \mathbb{R}^n$ is generated with i.i.d. entries from $\mathcal{N}(0,1)$.

Tables \ref{table:qp_lambda0} and \ref{table:qp_lambdaneq0} present the average numerical performance of each inexact APG method for solving problem \eqref{eq:QP_org} with $\lambda = 0$ (standard QP) and $\lambda > 0$ (sparse QP), respectively. To ensure a fair comparison, the results are averaged over 10 independent instances generated using distinct random seeds. The performance of each method is evaluated across a range of tolerance parameters. Specifically, for SpinAPG, iAPG-SLB/AIFB and o-iFB, we select $\Upsilon\in\{1, \,0.001\}$ and $p\in\{1.1, \,2.1, \,3.1\}$. For I-FISTA, we select $\tau\in\{0.1,0.3,0.5,0.7,0.9\}$ and $\alpha=0.001L$. In these tables, ``\texttt{kkt}" represents the final relative KKT residual, while ``\texttt{out}" and ``\texttt{inn}" record the number of outer APG iterations and the cumulative count of inner {\sc Ssn} iterations, respectively.

\begin{table}[ht]
\caption{Average numerical results of inexact APG methods with varying tolerance parameters for solving problem \eqref{eq:QP_org} with $n=200$ and $\lambda=0$.}
\label{table:qp_lambda0}
\renewcommand\arraystretch{0.9}
\centering \tabcolsep 6.5pt
\begin{tabular}{ccccc|ccccc}
     \toprule
     \multicolumn{5}{c|}{SpinAPG}&\multicolumn{5}{c}{iAPG-SLB/AIFB}   \\
     \cmidrule(lr){1-5} \cmidrule(lr){6-10}
     $(\Upsilon,p)$ & \texttt{kkt}  & \texttt{out} & \texttt{inn} & \texttt{time}
     & $(\Upsilon,p)$  & \texttt{kkt}  & \texttt{out} & \texttt{inn} & \texttt{time}   \\
     \cmidrule(lr){1-5} \cmidrule(lr){6-10}
     (1, 2.1) & 9.50e-07 & 465 & 1280 & 55 & (1, 2.1) & 7.96e-07 & 832 & 2007 & 968 \\
     (1, 3.1) & 9.87e-07 & 442 & 1226 & 52 & (1, 3.1) & 9.89e-07 & 439 & 1236 & 484 \\
     (0.001, 1.1) & 9.92e-07 & 435 & 1193 & 51 & (0.001, 1.1) & 8.21e-07 & 711 & 1780 & 862 \\
     (0.001, 2.1) & 9.93e-07 & 436 & 1271 & 57 & (0.001, 2.1) & 9.77e-07 & 436 & 1283 & 494 \\
     (0.001, 3.1) & 9.84e-07 & 433 & 1329 & 62 & (0.001, 3.1) & 9.90e-07 & 437 & 1386 & 493 \\
     
     \bottomrule \vspace{-2mm} \\
     
     \toprule
     \multicolumn{5}{c|}{o-iFB}&\multicolumn{5}{c}{I-FISTA}   \\
     \cmidrule(lr){1-5} \cmidrule(lr){6-10}
     $(\Upsilon,p)$ & \texttt{kkt}  & \texttt{out} & \texttt{inn} & \texttt{time}
     &$\tau$ & \texttt{kkt}  & \texttt{out} & \texttt{inn} & \texttt{time} \\
     \cmidrule(lr){1-5} \cmidrule(lr){6-10}
     (1, 2.1) & 9.14e-07 & 481 & 1201 & 213 & 0.1 & 9.87e-07 & 1484 & 3629 & 1197 \\
     (1, 3.1) & 9.77e-07 & 440 & 1213 & 147 & 0.3 & 9.95e-07 & 809 & 2073 & 747 \\
     (0.001, 1.1) & 9.91e-07 & 437 & 1176 & 155 & 0.5 & 9.94e-07 & 623 & 1635 & 609 \\
     (0.001, 2.1) & 9.92e-07 & 435 & 1257 & 133 & 0.7 & 9.91e-07 & 526 & 1415 & 541 \\
     (0.001, 3.1) & 9.90e-07 & 437 & 1325 & 115 & 0.9 & 9.81e-07 & 461 & 1295 & 496 \\
     \bottomrule
     \end{tabular}
\end{table}

\begin{table}[ht]
\caption{Same as Table \ref{table:qp_lambda0} but for $\lambda = 10\|\bm q_0\|_{\infty}$.
%Average numerical results of inexact APG methods with varying tolerance parameters for solving problem \eqref{eq:QP_org} with $n=200$ and $\lambda = 10\|\bm q_0\|_{\infty}$.
The resulting approximate solutions have a sparsity level of about 58\%.}
\label{table:qp_lambdaneq0}
\renewcommand\arraystretch{0.9}
\centering \tabcolsep 6.5pt
\begin{tabular}{ccccc|ccccc}
     \toprule
     \multicolumn{5}{c|}{SpinAPG}&\multicolumn{5}{c}{iAPG-SLB/AIFB}   \\
     \cmidrule(lr){1-5} \cmidrule(lr){6-10}
     $(\Upsilon,p)$ & \texttt{kkt}  & \texttt{out} & \texttt{inn} & \texttt{time}
     & $(\Upsilon,p)$  & \texttt{kkt}  & \texttt{out} & \texttt{inn} & \texttt{time}   \\
     \cmidrule(lr){1-5} \cmidrule(lr){6-10}
     (1, 2.1) & 5.65e-07 & 395 & 1730 & 91 & (1, 2.1) & 6.46e-07 & 585 & 1904 & 783 \\
     (1, 3.1) & 7.71e-07 & 146 & 762 & 42 & (1, 3.1) & 8.54e-07 & 160 & 754 & 307 \\
     (0.001, 1.1) & 8.39e-07 & 167 & 821 & 45 & (0.001, 1.1) & 9.01e-07 & 320 & 1157 & 479 \\
     (0.001, 2.1) & 9.38e-07 & 157 & 806 & 44 & (0.001, 2.1) & 8.37e-07 & 152 & 785 & 308 \\
     (0.001, 3.1) & 9.21e-07 & 152 & 828 & 46 & (0.001, 3.1) & 8.47e-07 & 153 & 861 & 314 \\
     
     \bottomrule \vspace{-2mm} \\
     
     \toprule
     \multicolumn{5}{c|}{o-iFB}&\multicolumn{5}{c}{I-FISTA}   \\
     \cmidrule(lr){1-5} \cmidrule(lr){6-10}
     $(\Upsilon,p)$ & \texttt{kkt}  & \texttt{out} & \texttt{inn} & \texttt{time}
     &$\tau$ & \texttt{kkt}  & \texttt{out} & \texttt{inn} & \texttt{time} \\
     \cmidrule(lr){1-5} \cmidrule(lr){6-10}
     (1, 2.1) & 7.26e-07 & 446 & 1833 & 259 & 0.1 & 8.20e-07 & 554 & 2170 & 773 \\
     (1, 3.1) & 8.99e-07 & 156 & 766 & 85 & 0.3 & 9.22e-07 & 301 & 1294 & 484 \\
     (0.001, 1.1) & 8.59e-07 & 193 & 853 & 105 & 0.5 & 7.80e-07 & 222 & 1019 & 389 \\
     (0.001, 2.1) & 8.19e-07 & 150 & 784 & 79 & 0.7 & 8.58e-07 & 178 & 879 & 341 \\
     (0.001, 3.1) & 8.46e-07 & 153 & 818 & 72 & 0.9 & 8.33e-07 & 161 & 820 & 313 \\
     \bottomrule
     \end{tabular}
\end{table}

From Tables \ref{table:qp_lambda0} and \ref{table:qp_lambdaneq0}, SpinAPG consistently attains the prescribed KKT tolerance with the shortest running time. Notably, although the total numbers of outer and inner iterations are sometimes comparable across different inexact methods, their computational costs differ substantially. This performance discrepancy is consistent with our analysis in Sections~\ref{sec-example} and~\ref{sec:implement_APG}. In particular, all inexact APG methods except SpinAPG require the explicit computation of a feasible point $\widetilde{\bm x}^{k,t}$ during updates, which introduces non-negligible overhead even for the moderate-sized problems considered here. Moreover, iAPG-SLB and AIFB require the evaluation of both primal and dual objective values, which further complicates the verification process when either objective function is not readily accessible. In contrast, I-FISTA employs a relative-type error criterion that eliminates the need to prescribe a summable tolerance sequence and is therefore potentially more tuning-friendly. However, this advantage comes at the cost of a more involved verification procedure, as discussed in Sections~\ref{sec-example} and~\ref{sec:implement_APG}. In addition, as noted in \cite[Section 4]{bgk2023on}, choosing a small $\tau<1$ effectively reduces the standard step size $1/L$, which may further slow the practical convergence of I-FISTA. %This likely explains why I-FISTA requires more computational time to reach the required accuracy.

In addition, the performance of each inexact APG method is unsurprisingly affected by the choice of tolerance parameters, namely, $(\Upsilon,p)$ for SpinAPG, iAPG-SLB/AIFB, o-iFB, and $\tau$ for I-FISTA, as these parameters directly control the inexactness thresholds. Generally, to reach a prescribed accuracy, a larger $\Upsilon$ with a smaller $p$, or a smaller $\tau$, leads to more outer iterations since the error criterion \eqref{stopcond-qp-SpinAPG}, \eqref{stopcond-qp-slb} or \eqref{stopcond-qp-ifista} becomes easier to satisfy. However, this does not necessarily reduce computational time. As observed in Tables \ref{table:qp_lambda0} and \ref{table:qp_lambdaneq0}, choosing $(\Upsilon,p)=(1, 2.1)$ or $\tau=0.1$ consistently results in the highest number of outer iterations but leads to the longest running time. This behavior aligns to some extent with the theoretical results in \cite{ad2015stability,bgk2023on,srb2011convergence,vsbv2013accelerated} and this work, which guarantee convergence (rates) of inexact APG methods only under certain error conditions, though these conditions appear to be conservative in practice. Conversely, setting $(\Upsilon,p)=(0.001, 3.1)$ leads to the fastest tolerance decay but does not necessarily produce the shortest running time for SpinAPG and iAPG-SLB/AIFB, primarily due to the excessive cost of solving each subproblem more accurately. This suggests that enforcing a tighter tolerance for solving subproblems does not necessarily improve overall efficiency, highlighting the trade-off between the number of outer and inner iterations.

To further evaluate the performance of inexact APG methods across varying problem dimensions and to demonstrate the potential advantages of SpinAPG in more general settings, we conduct additional experiments with problem sizes ranging from $n=50$ to $n=300$. Both the standard QP case ($\lambda = 0$) and the sparse variant ($\lambda > 0$) of \eqref{eq:QP_org} are considered. For each dimension $n$, we generate $10$ independent instances using different random seeds. The average numerical performance, illustrating the scaling behavior of each method, is presented in Figures \ref{figure:qp_lambda0} and \ref{figure:qp_lambdaneq0}. The tolerance parameters are chosen based on the performance shown in Tables \ref{table:qp_lambda0} and \ref{table:qp_lambdaneq0}. Specifically, for SpinAPG and iAPG-SLB/AIFB, we set $(\Upsilon,p)=(1,3.1)$; for o-iFB, we set $(\Upsilon,p)=(0.001,3.1)$; and for I-FISTA, we set $\tau=0.9$ and $\alpha=0.001L$.

\begin{figure}[ht]
\begin{minipage}[!]{0.48\linewidth}
\centering
\hfill
\renewcommand{\arraystretch}{1.1}
\tabcolsep 5pt
\scalebox{0.96}{
\begin{tabular}{c|cccc}
\toprule
 & \multicolumn{4}{c}{\texttt{kkt}} \\
\cmidrule(lr){2-5}
$n$ & \texttt{a} & \texttt{b} & \texttt{c} & \texttt{d} \\ \midrule
50 & 9.38e-07 & 9.61e-07 & 9.63e-07 & 9.39e-07 \\
100 & 9.56e-07 & 9.38e-07 & 9.71e-07 & 9.71e-07 \\
150 & 9.80e-07 & 9.84e-07 & 9.87e-07 & 9.90e-07 \\
200 & 9.86e-07 & 9.89e-07 & 9.66e-07 & 9.76e-07 \\
250 & 9.88e-07 & 9.87e-07 & 9.82e-07 & 9.84e-07 \\
300 & 9.88e-07 & 9.74e-07 & 9.71e-07 & 9.93e-07 \\
\bottomrule
\end{tabular}}
\end{minipage}
\quad
\begin{minipage}[!]{0.48\linewidth}
\centering
\includegraphics[width=0.9\textwidth]{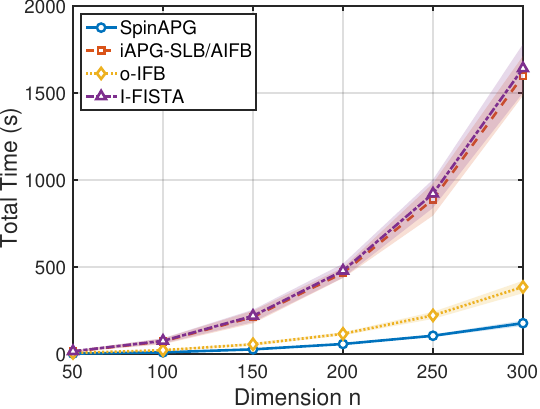}
\end{minipage}
%\hfill\vspace{2mm}
\caption{Average numerical results of SpinAPG (\texttt{a}), iAPG-SLB/AIFB (\texttt{b}), o-iFB (\texttt{c}), and I-FISTA (\texttt{d}) for solving problem \eqref{eq:QP_org} ($\lambda =0$) with $n$ ranging from $50$ to $300$. The curves show the averaged runtime, and the shaded regions represent one standard deviation.}\label{figure:qp_lambda0}
\end{figure}

\begin{figure}[ht]
\begin{minipage}[!]{0.48\linewidth}
\centering
\hfill
\renewcommand{\arraystretch}{1.1}
\tabcolsep 5pt
\scalebox{0.96}{
\begin{tabular}{c|cccc}
\toprule
 & \multicolumn{4}{c}{\texttt{kkt}} \\
\cmidrule(lr){2-5}
$n$ & \texttt{a} & \texttt{b} & \texttt{c} & \texttt{d} \\ \midrule
50 & 3.01e-07 & 2.24e-07 & 1.94e-07 & 2.26e-07 \\
100 & 6.13e-07 & 6.87e-07 & 4.81e-07 & 5.05e-07 \\
150 & 7.81e-07 & 7.09e-07 & 8.66e-07 & 8.15e-07 \\
200 & 8.51e-07 & 8.36e-07 & 8.07e-07 & 8.13e-07 \\
250 & 7.77e-07 & 7.25e-07 & 7.57e-07 & 7.66e-07 \\
300 & 9.25e-07 & 9.08e-07 & 9.11e-07 & 8.72e-07 \\
\bottomrule
\end{tabular}}
\end{minipage}
\quad
\begin{minipage}[!]{0.48\linewidth}
\centering
\includegraphics[width=0.9\textwidth]{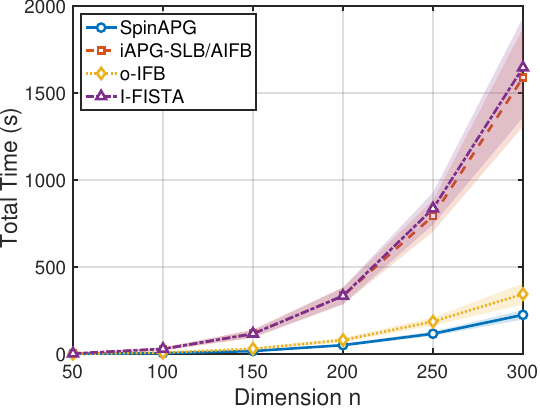}
\end{minipage}
%\hfill\vspace{2mm}
\caption{Same as Figure \ref{figure:qp_lambda0} but for $\lambda =10 \|\bm{q}_0\|_{\infty}$.
%Average numerical results of SpinAPG (\texttt{a}), iAPG-SLB/AIFB (\texttt{b}), o-iFB (\texttt{c}), and I-FISTA (\texttt{d}) for solving problem \eqref{eq:QP_org} ($\lambda =10 \|\bm{q}_0\|_{\infty}$) with $n$ ranging from $50$ to $300$.
}\label{figure:qp_lambdaneq0}
\end{figure}

As illustrated in Figures \ref{figure:qp_lambda0} and \ref{figure:qp_lambdaneq0}, SpinAPG consistently outperforms all other compared methods across all tested dimensions, exhibiting the fastest convergence in terms of total running time. Notably, for the projection onto $\mathrm{dom}\,P$, we have employed a highly efficient {\sc Ssn} to obtain a feasible point $\widetilde{\bm x}^{k,t}$. As previously mentioned, this {\sc Ssn} solver is significantly more efficient than standard alternatives such as {\sc Gurobi} and {\sc Matlab}'s built-in \texttt{quadprog}. Even so, the additional overhead incurred by iAPG-SLB, AIFB, o-iFB, and I-FISTA for computing feasible iterates becomes increasingly substantial as the problem size increases. This burden would be even greater in the absence of an efficient projection procedure, particularly when dealing with more complex feasible sets.

\section{Concluding remarks}\label{seccon}

In this work, we develop SpinAPG, a new inexact variant of the APG method, which employs a shadow-point enhanced error criterion that allows APG to operate with potentially infeasible iterates without explicitly constructing feasible ones, while remaining adaptable to various types of inexact errors when solving the subproblem. The proposed framework accommodates a broader range of extrapolation parameters and preserves all desirable convergence properties of the APG method, including the improved $o(1/k^2)$ convergence rate and the iterate convergence. Numerical experiments on sparse quadratic programming problems demonstrate its computational efficiency and practical advantages. Future work may explore extensions of the shadow-point enhanced error criterion to nonconvex problems and variable-metric settings, thereby broadening its applicability to more general optimization scenarios.

%%%%%%%%%%%%%%%%%%%%%%%%%%%%%%%%%%%%%%%%%%%%%%%%%%%%%
\begin{appendices}

%%%%%%%%%%%%%%%%%%%%%%%%%%%%%%%%%%%%%%%%%%%
\section{Additional proofs in Section \ref{sec-convana}}

%%%%%%%%%%%%%%%%%%%%%%%%%%%%%%%%%%%%%%%%%%%%%%%
\subsection{Proof of Lemma \ref{lem-thetasumbd}}\label{proof-lem-thetasumbd}

We first show by induction that $\theta_k\geq\frac{1}{k+ 2 }$ for all $k\geq0$. This obviously holds for $k=0$ since $\theta_0=1$. Suppose that $\theta_k\geq\frac{1}{k+ 2 }$ holds for some $k\geq0$. Since $0<\theta_k\leq1$, then $(1-\theta_{k})^{ 2 }\leq1-\theta_{k}$. Using this and \eqref{condtheta}, we see that $\frac{(1-\theta_{k+1})^{2}}{\theta_{k+1}^{2}}\leq\frac{1}{\theta_{k}^{2}}$ and hence $\frac{1-\theta_{k+1}}{\theta_{k+1}} \leq \frac{1}{\theta_{k}}$. This then implies that $\theta_{k+1}\geq\frac{\theta_k}{\theta_k+1}\geq\frac{1}{k+3}$. Thus, we complete the induction. Using this fact, we further obtain that
\begin{equation*}
\begin{aligned}
\sum^{k}_{i=0}\frac{1}{\theta_{i-1}^{ 2 }} \!-\! \frac{1-\theta_{i}}{\theta_{i}^{ 2 }}
= \frac{1}{\theta_{-1}^{ 2 }} - \frac{1}{\theta_{k}^{ 2 }}
+ \sum^{k}_{i=0}\frac{1}{\theta_{i}}
\leq 2 + \sum^{k}_{i=1}\frac{1}{\theta_{i}}
\leq 2 + \sum^{k}_{i=1}(i+2)
\leq 2 + \frac{(k+3 )^{ 2 }}{ 2 }.
\end{aligned}
\end{equation*}
This completes the proof.

%%%%%%%%%%%%%%%%%%%%%%%%%%%%%%%%%%%%%%%%%%%%%%
\subsection{Proof of Lemma \ref{lem-recursion}}\label{proof-lem-recursion}

For any $k\geq0$, set $A_k:=\max_{0 \leq i \leq k}\big\{a_i\big\}$, $Q_k:=\max_{0\leq i \leq k}\big\{q_i\big\}$ and $P_k := \sum^k_{i=0}\big(\lambda_i+\widetilde{\lambda}_{i}\big)$. Then, we see from the recursion that, for any $0\leq \ell \leq k$,
\begin{equation*}
a_{\ell+1}^2
\leq q_\ell + {\textstyle\sum^\ell_{i=0}}\big(\lambda_ia_{i+1}+\widetilde{\lambda}_{i}a_i+c_i\big)
%\leq Q_k + {\textstyle\sum^k_{i=0}}\big(\lambda_ia_{i+1}+\widetilde{\lambda}_{i}a_i+c_i\big) \\
%&\leq Q_k + {\textstyle\sum^k_{i=0}}c_i
%+ A_{k+1}{\textstyle\sum^k_{i=0}}\big(\lambda_i+\widetilde{\lambda}_{i}\big)
\leq Q_k + {\textstyle\sum^k_{i=0}}c_i + A_{k+1} P_k,
\end{equation*}
Taking the maximum over $0\leq \ell \leq k$ results in $A_{k+1}^2 \leq Q_k + {\textstyle\sum^k_{i=0}}c_i + A_{k+1} P_{k}$.
%\begin{equation*}
%A_{k+1}^2 \leq Q_k + {\textstyle\sum^k_{i=0}}c_i + A_{k+1} P_{k}.
%%{\textstyle\sum^k_{i=0}}\big(\lambda_i+\widetilde{\lambda}_{i}\big).
%\end{equation*}
Consequently, bounding by the roots of the quadratic equation, we have $A_{k+1} \leq \frac{1}{2}P_k + \sqrt{Q_k + {\textstyle\sum^k_{i=0}} c_i+\big({\textstyle\frac{1}{2}} P_k\big)^2}$,
%\begin{equation*}
%A_{k+1}
%\leq \frac{1}{2}P_k + \sqrt{Q_k + {\textstyle\sum^k_{i=0}} c_i+\big({\textstyle\frac{1}{2}} P_k\big)^2},
%\end{equation*}
which, together with $a_{k+1}\leq A_{k+1}$, proves the first result. Moreover, for any $0\leq i \leq k$,
\begin{equation*}
\begin{aligned}
&\quad \lambda_i a_{i+1} + \widetilde{\lambda}_{i}a_i
\leq (\lambda_i+\widetilde{\lambda}_{i})\,A_{i+1} \\
&\leq {\textstyle\frac{1}{2}}(\lambda_i+\widetilde{\lambda}_{i})
P_i + (\lambda_i+\widetilde{\lambda}_{i})\sqrt{Q_i + {\textstyle\sum^i_{j=0}} c_j+\big(\textstyle{\frac{1}{2}}P_i\big)^2} \\
&\leq (\lambda_i+\widetilde{\lambda}_{i}){\textstyle\sum^i_{j=0}}
\big(\lambda_j+\widetilde{\lambda}_{j}\big) + (\lambda_i+\widetilde{\lambda}_{i})\big(Q_i^{\frac{1}{2}} + \big({\textstyle\sum^i_{j=0}}c_j\big)^{\frac{1}{2}}\big).
\end{aligned}
\end{equation*}
Thus, summing this inequality from $i=0$ to $i=k$, we have
\begin{equation*}
\begin{aligned}
&\quad {\textstyle\sum^k_{i=0}}\big(\lambda_ia_{i+1}+\widetilde{\lambda}_{i}a_i\big) \\
&\leq {\textstyle\sum^k_{i=0}}\big((\lambda_i+\widetilde{\lambda}_{i})
{\textstyle\sum^i_{j=0}}\big(\lambda_j+\widetilde{\lambda}_{j}\big)\big)
+ {\textstyle\sum^k_{i=0}}(\lambda_i+\widetilde{\lambda}_{i})Q_i^{\frac{1}{2}}
+ {\textstyle\sum^k_{i=0}}(\lambda_i+\widetilde{\lambda}_{i})
\big({\textstyle\sum^i_{j=0}}c_j\big)^{\frac{1}{2}} \\
&\leq \big({\textstyle\sum^k_{i=0}}(\lambda_i+\widetilde{\lambda}_{i})\big)^2
+ {\textstyle\sum^k_{i=0}}(\lambda_i+\widetilde{\lambda}_{i})Q_i^{\frac{1}{2}}
+ \big({\textstyle\sum^k_{i=0}}(\lambda_i+\widetilde{\lambda}_{i})\big)
\big({\textstyle\sum^k_{i=0}}c_i\big)^{\frac{1}{2}}.
\end{aligned}
\end{equation*}
From here, the second result follows.

%%%%%%%%%%%%%%%%%%%%%%%%%%%%%%%%%%%%%%%%%%%%%%
\subsection{Proof of Lemma \ref{lem-appsuffdes}}\label{proof-lem-appsuffdes}

First, from condition \eqref{inexcond-iAPG}, there exists $\bm{d}^{k+1}\in \partial_{\varepsilon_k}P(\widetilde{\bm{x}}^{k+1})$ such that $\Delta^k=\bm{d}^{k+1} + \nabla f(\bm{y}^k) + L(\bm{x}^{k+1}-\bm{y}^{k})$. Then, for any $\bm{x}\in\mathrm{dom}\,P$, we see that
\begin{equation*}
%\begin{aligned}
P(\bm{x})
%&\geq P(\widetilde{\bm{x}}^{k+1}) + \langle\bm{d}^{k+1}, \,\bm{x} - \widetilde{\bm{x}}^{k+1}\rangle - \varepsilon_k  \\
\geq P(\widetilde{\bm{x}}^{k+1}) + \langle\Delta^k-\nabla f(\bm{y}^k)-L(\bm{x}^{k+1}-\bm{y}^{k}), \,\bm{x}-\widetilde{\bm{x}}^{k+1}\rangle
- \varepsilon_k,
%\end{aligned}
\end{equation*}
which, together with $\|\Delta^k\|\leq\eta_k$ and $\varepsilon_k\leq\nu_k$ in criterion \eqref{inexcond-iAPG}, implies that
\begin{equation*}
\begin{aligned}
P(\widetilde{\bm{x}}^{k+1})
%&\leq P(\bm{x}) - \langle \nabla f(\bm{y}^k), \,\widetilde{\bm{x}}^{k+1} - \bm{x} \rangle
%+ L\langle\bm{x}^{k+1}-\bm{y}^k, \,\bm{x}-\widetilde{\bm{x}}^{k+1}\rangle
%+ \langle\Delta^k,\,\widetilde{\bm{x}}^{k+1}-\bm{x}\rangle + \varepsilon_k \\
&\leq P(\bm{x}) - \langle \nabla f(\bm{y}^k), \,\widetilde{\bm{x}}^{k+1} - \bm{x} \rangle
+ L\langle\bm{x}^{k+1}-\bm{y}^k, \,\bm{x}-\widetilde{\bm{x}}^{k+1}\rangle  
+ \eta_k\|\widetilde{\bm{x}}^{k+1}-\bm{x}\| + \nu_k.
\end{aligned}
\end{equation*}
%where the last inequality follows from $\|\Delta^k\|\leq\eta_k$ and $\varepsilon_k\leq\nu_k$ in condition \eqref{inexcond-iAPG}.
Moreover, one can verify that
\begin{equation*}
\langle\bm{x}^{k+1}-\bm{y}^k, \,\bm{x}-\widetilde{\bm{x}}^{k+1}\rangle
= {\textstyle\frac{1}{2}}\|\bm{x}-\bm{y}^k\|^2
- {\textstyle\frac{1}{2}}\|\bm{x}-\bm{x}^{k+1}\|^2
- {\textstyle\frac{1}{2}}\|\widetilde{\bm{x}}^{k+1}-\bm{y}^k\|^2
+ {\textstyle\frac{1}{2}}\|\widetilde{\bm{x}}^{k+1}-\bm{x}^{k+1}\|^2.
\end{equation*}
Combining the above two inequalities, together with $\|\widetilde{\bm{x}}^{k+1}-\bm{x}^{k+1}\|\leq\mu_k$ in \eqref{inexcond-iAPG}, we obtain
\begin{equation}\label{ineq1-gen}
\begin{aligned}
P(\widetilde{\bm{x}}^{k+1})
&\leq P(\bm{x})
- \langle \nabla f(\bm{y}^k), \,\widetilde{\bm{x}}^{k+1} - \bm{x} \rangle
+ {\textstyle\frac{L}{2}}\|\bm{x}-\bm{y}^k\|^2
- {\textstyle\frac{L}{2}}\|\bm{x}-\bm{x}^{k+1}\|^2  \\
&\qquad
- {\textstyle\frac{L}{2}}\|\widetilde{\bm{x}}^{k+1}-\bm{y}^k\|^2
+ \eta_k\|\widetilde{\bm{x}}^{k+1}-\bm{x}\|
+ {\textstyle\frac{L}{2}}\mu_k^2 + \nu_k.
\end{aligned}
\end{equation}
On the other hand, since $f$ is convex and continuously differentiable with a Lipschitz continuous gradient (by Assumption \ref{assumA}2), then $f(\bm{x})
\geq f(\bm{y}^k) + \langle \nabla f(\bm{y}^k), \,\bm{x}-\bm{y}^k\rangle$ and
\begin{equation*}
\begin{aligned}
f(\widetilde{\bm{x}}^{k+1})
\leq f(\bm{y}^k)
+ \langle \nabla f(\bm{y}^k), \,\widetilde{\bm{x}}^{k+1}-\bm{y}^k\rangle
+ {\textstyle\frac{L}{2}}\|\widetilde{\bm{x}}^{k+1}-\bm{y}^k\|^2.
%f(\bm{x})
%& \geq f(\bm{y}^k) + \langle \nabla f(\bm{y}^k), \,\bm{x}-\bm{y}^k\rangle.
\end{aligned}
\end{equation*}
Summing these two inequalities yields that
\begin{equation}\label{ineq2-gen}
f(\widetilde{\bm{x}}^{k+1})
\leq f(\bm{x}) + \langle \nabla f(\bm{y}^k), \,\widetilde{\bm{x}}^{k+1}-\bm{x}\rangle + {\textstyle\frac{L}{2}}\|\widetilde{\bm{x}}^{k+1}-\bm{y}^k\|^2.
\end{equation}
Thus, summing \eqref{ineq1-gen} and \eqref{ineq2-gen}, we obtain the desired result.

%%%%%%%%%%%%%%%%%%%%%%%%%%%%%%%%%%%%%%%%%%%%%%
\subsection{Proof of Lemma \ref{lem-suffdes-iAPG}}\label{proof-lem-suffdes-iAPG}

First, since $\widetilde{\bm{x}}^k\in\mathrm{dom}\,P$, $\mathrm{dom}\,P$ is convex and $\theta_{k}\in(0,1]$, then $(1-\theta_k)\widetilde{\bm{x}}^k+\theta_k\bm{x}\in\mathrm{dom}\,P$ for any $\bm{x}\in\mathrm{dom}\,P$. Setting $\bm{x}=(1-\theta_k)\widetilde{\bm{x}}^k+\theta_k\bm{x}$ in \eqref{suffdes1-iAPG} and using the convexity of $F$, we have
\begin{equation}\label{Fineq1-iAPG}
\begin{aligned}
F(\widetilde{\bm{x}}^{k+1})
&\leq F\big((1-\theta_k)\widetilde{\bm{x}}^k+\theta_k\bm{x}\big) \\
&\quad + {\textstyle\frac{L}{2}}\|(1-\theta_k)\widetilde{\bm{x}}^k+\theta_k\bm{x}-\bm{y}^{k}\|^2 - {\textstyle\frac{L}{2}}\|(1-\theta_k)\widetilde{\bm{x}}^k+\theta_k\bm{x}-\bm{x}^{k+1}\|^2 \\
&\quad + \eta_k\|\widetilde{\bm{x}}^{k+1} - (1-\theta_k)\widetilde{\bm{x}}^k-\theta_k\bm{x}\|
+ {\textstyle\frac{L}{2}}\mu_k^2 + \nu_k \\
&\leq (1-\theta_k)F(\widetilde{\bm{x}}^k) + \theta_k F(\bm{x}) \\
&\quad + {\textstyle\frac{L}{2}}\|(1-\theta_k)\widetilde{\bm{x}}^k+\theta_k\bm{x}-\bm{y}^{k}\|^2
- {\textstyle\frac{L}{2}}\|(1-\theta_k)\widetilde{\bm{x}}^k+\theta_k\bm{x}-\bm{x}^{k+1}\|^2 \\
&\quad + \eta_k\|\widetilde{\bm{x}}^{k+1} - (1-\theta_k)\widetilde{\bm{x}}^k-\theta_k\bm{x}\|
+ {\textstyle\frac{L}{2}}\mu_k^2 + \nu_k.
\end{aligned}
\end{equation}
%where the last inequality follows from the convexity of $F$.
Note that
\begin{equation}\label{Fineq3-iAPG}
\begin{aligned}
&\quad {\textstyle\frac{L}{2}}\|(1-\theta_k)\widetilde{\bm{x}}^k+\theta_k\bm{x}-\bm{y}^k\|^2 -{\textstyle\frac{L}{2}}\|(1-\theta_k)\widetilde{\bm{x}}^k+\theta_k\bm{x}-\bm{x}^{k+1}\|^2 \\
%&=
%{\textstyle\frac{L}{2}}\|(1-\theta_k)(\widetilde{\bm{x}}^k-\bm{x}^k)
%+ (1-\theta_k)\bm{x}^k + \theta_k\bm{x}-\bm{y}^k\|^2 \\
%&\quad -{\textstyle\frac{L}{2}}\|(1-\theta_k)(\widetilde{\bm{x}}^k-\bm{x}^k)
%+ (1-\theta_k)\bm{x}^k + \theta_k\bm{x}-\bm{x}^{k+1}\|^2  \\
&=
{\textstyle\frac{L}{2}}\theta_k^2\|\bm{x}
- \big(\theta_k^{-1}\bm{y}^k - (\theta_k^{-1}-1)\bm{x}^k\big)
+ (\theta_k^{-1}-1)(\widetilde{\bm{x}}^k-\bm{x}^k)\|^2 \\
&\qquad - {\textstyle\frac{L}{2}}\theta_k^2\|\bm{x}
- \big(\bm{x}^{k+1} + (\theta_k^{-1}-1)(\bm{x}^{k+1}-\bm{x}^k)\big)
+ (\theta_k^{-1}-1)(\widetilde{\bm{x}}^k-\bm{x}^k)\|^2 \\
&=
{\textstyle\frac{L}{2}}\theta_k^2\|\bm{x} \!-\! \bm{z}^k + (\theta_k^{-1}\!-\!1)(\widetilde{\bm{x}}^k-\bm{x}^k)\|^2
- {\textstyle\frac{L}{2}}\theta_k^2\|\bm{x} \!-\! \bm{z}^{k+1} + (\theta_k^{-1}\!-\!1)(\widetilde{\bm{x}}^k-\bm{x}^k)\|^2 \\[3pt]
%&= {\textstyle\frac{L}{2}}\theta_k^2\|\bm{x} - \bm{z}^k\|^2
%+ L\theta_k(1-\theta_k)\langle\bm{x}-\bm{z}^k, %\,\widetilde{\bm{x}}^k-\bm{x}^k\rangle
%+ {\textstyle\frac{L}{2}}(1-\theta_k)^2\|\widetilde{\bm{x}}^k-\bm{x}^k\|^2 \\
%&\qquad -{\textstyle\frac{L}{2}}\theta_k^2\|\bm{x} - \bm{z}^{k+1}\|^2
%- L\theta_k(1-\theta_k)\langle\bm{x}-\bm{z}^{k+1}, %\,\widetilde{\bm{x}}^k-\bm{x}^k\rangle
%- {\textstyle\frac{L}{2}}(1-\theta_k)^2\|\widetilde{\bm{x}}^k-\bm{x}^k\|^2  \\
&=
{\textstyle\frac{L}{2}}\theta_k^2\|\bm{x} - \bm{z}^k\|^2 -{\textstyle\frac{L}{2}}\theta_k^2\|\bm{x} - \bm{z}^{k+1}\|^2
+ L\theta_k(1-\theta_k)\langle\bm{z}^{k+1}-\bm{z}^k, \,\widetilde{\bm{x}}^k-\bm{x}^k\rangle \\[3pt]
&\leq {\textstyle\frac{L}{2}}\theta_k^2\|\bm{x} - \bm{z}^k\|^2 -{\textstyle\frac{L}{2}}\theta_k^2\|\bm{x} - \bm{z}^{k+1}\|^2
+ L\theta_k(1-\theta_k)\mu_{k-1}\|\bm{z}^{k+1}-\bm{z}^k\| \\
&\leq {\textstyle\frac{L}{2}}\theta_k^2\|\bm{x} \!-\! \bm{z}^k\|^2 \!-\! {\textstyle\frac{L}{2}}\theta_k^2\|\bm{x} \!-\! \bm{z}^{k+1}\|^2
\!+\! L\theta_k(1\!-\!\theta_k)\mu_{k-1}\big(\|\bm{x}\!-\!\bm{z}^{k+1}\|
\!+\! \|\bm{x}\!-\!\bm{z}^k\|\big),
\end{aligned}
\end{equation}
where the second equality follows from $\bm{z}^k = \bm{x}^k + (\theta_{k-1}^{-1}-1)(\bm{x}^k-\bm{x}^{k-1}) = \theta_k^{-1}\bm{y}^{k} - (\theta_k^{-1}-1)\bm{x}^k$. Moreover,
\begin{equation}\label{Fineq4-iAPG}
\begin{aligned}
&\quad \|\widetilde{\bm{x}}^{k+1} - (1-\theta_k)\widetilde{\bm{x}}^k-\theta_k\bm{x}\|  \\
&= \|\big(\bm{x}^{k+1} - (1-\theta_k)\bm{x}^k - \theta_k\bm{x}\big)
+ \big(\widetilde{\bm{x}}^{k+1}-\bm{x}^{k+1}\big)
- (1-\theta_k)(\widetilde{\bm{x}}^{k}-\bm{x}^{k})\| \\
&= \|\theta_k\big(\bm{x}^{k+1} \!+\! (\theta_{k}^{-1}\!-\!1)(\bm{x}^{k+1}\!-\!\bm{x}^{k})\!-\!\bm{x}\big)
+ \big(\widetilde{\bm{x}}^{k+1}\!-\!\bm{x}^{k+1}\big)
- (1\!-\!\theta_k)(\widetilde{\bm{x}}^{k}\!-\!\bm{x}^{k})\| \\
&= \|\theta_k\big(\bm{z}^{k+1} - \bm{x}\big)
+ \big(\widetilde{\bm{x}}^{k+1}-\bm{x}^{k+1}\big)
- (1-\theta_k)(\widetilde{\bm{x}}^{k}-\bm{x}^{k})\| \\
%&\leq \theta_k\|\bm{x}-\bm{z}^{k+1}\| + \mu_k + (1-\theta_k)\mu_{k-1} \\
&\leq \theta_k\|\bm{x}-\bm{z}^{k+1}\| + \mu_k + \mu_{k-1},
\end{aligned}
\end{equation}
where the inequality follows from criterion \eqref{inexcond-iAPG} and $\theta_k\in(0,1]$. Thus, substituting \eqref{Fineq3-iAPG} and \eqref{Fineq4-iAPG} into \eqref{Fineq1-iAPG}, we obtain that
\begin{equation*}%\label{Fineq-iAPG}
\begin{aligned}
F(\widetilde{\bm{x}}^{k+1})
&\leq (1-\theta_k)F(\widetilde{\bm{x}}^k) + \theta_k F(\bm{x})
+ {\textstyle\frac{L}{2}}\theta_k^2\|\bm{x} - \bm{z}^k\|^2
- {\textstyle\frac{L}{2}}\theta_k^2\|\bm{x} - \bm{z}^{k+1}\|^2  \\
&\quad
+ \left(\theta_k\eta_k+L\theta_k(1-\theta_k)\mu_{k-1}\right)\|\bm{x}-\bm{z}^{k+1}\|
+ L\theta_k(1-\theta_k)\mu_{k-1}\|\bm{x}-\bm{z}^k\|  \\
&\quad + {\textstyle\frac{L}{2}}\mu_k^2
+ \eta_k(\mu_k+\mu_{k-1}) + \nu_k.
\end{aligned}
\end{equation*}
Now, in the above inequality, subtracting $F(\bm{x})$ from both sides, dividing both sides by $\theta_k^2$ and then rearranging the resulting relation, we obtain that \begin{equation*}
\begin{aligned}
&\quad \theta_{k}^{-2}\big(F(\widetilde{\bm{x}}^{k+1})-F(\bm{x})\big)
+ {\textstyle\frac{L}{2}}\|\bm{x}-\bm{z}^{k+1}\|^2  \\
&\leq \theta_{k}^{-2}(1-\theta_k)\big(F(\widetilde{\bm{x}}^{k})-F(\bm{x})\big)
+ {\textstyle\frac{L}{2}}\|\bm{x}-\bm{z}^k\|^2 \\
&\quad + \left(\theta_k^{-1}\eta_k+L\theta_k^{-1}(1-\theta_k)\mu_{k-1}\right)
\|\bm{x}-\bm{z}^{k+1}\|
+ L\theta_k^{-1}(1-\theta_k)\mu_{k-1}\|\bm{x}-\bm{z}^k\| + \xi_k,
\end{aligned}
\end{equation*}
where $\xi_k:=\theta_k^{-2}\big(\frac{L}{2}\mu_k^2+\eta_k(\mu_k+\mu_{k-1})+\nu_k\big)$.
Moreover, for $k=0$, we have that $\theta_0^{-1}(1-\theta_0)=0<(\alpha-1)\theta_{-1}^{-1}$, and for $k\geq1$, we have that
\begin{equation*}
\theta_k^{-1}(1-\theta_k)\leq\theta_k\theta_{k-1}^{-2}
\leq\frac{(\alpha-1)(k+1)}{k+\alpha-1}
\theta_{k-1}^{-1}\leq(\alpha-1)\theta_{k-1}^{-1},
\end{equation*}
where the first inequality follows from \eqref{condtheta} and the second inequality follows from \eqref{condtheta1} and $\theta_{k-1}\geq\frac{1}{k+1}$ (by Lemma \ref{lem-thetasumbd}). This, together with the above inequality, implies that
\begin{equation}\label{suffdes-iAPG-tmp}
\begin{aligned}
&~~ \theta_{k}^{-2}\big(F(\widetilde{\bm{x}}^{k+1})-F(\bm{x})\big)
+ {\textstyle\frac{L}{2}}\|\bm{x} - \bm{z}^{k+1}\|^2 \\
&\leq \theta_k^{-2}(1-\theta_k)\big(F(\widetilde{\bm{x}}^{k})-F(\bm{x})\big)
+ {\textstyle\frac{L}{2}}\|\bm{x} - \bm{z}^k\|^2 \\
&\qquad
+ \big(\theta_k^{-1}\eta_k+(\alpha-1)L\theta_{k-1}^{-1}\mu_{k-1}\big)\|\bm{x}-\bm{z}^{k+1}\| + (\alpha-1)L\theta_{k-1}^{-1}\mu_{k-1}\|\bm{x}-\bm{z}^k\| + \xi_k.
\end{aligned}
\end{equation}
Moreover,
\begin{equation*}
\begin{aligned}
\theta_{k}^{-2}\big(F(\widetilde{\bm{x}}^{k+1})-F(\bm{x})\big)
&= \theta_{k+1}^{-2}(1-\theta_{k+1})\big(F(\widetilde{\bm{x}}^{k+1})-F(\bm{x})\big)
+ \vartheta_{k+1}\big(F(\widetilde{\bm{x}}^{k+1})-F(\bm{x})\big) \\
&\geq
\theta_{k+1}^{-2}(1-\theta_{k+1})\big(F(\widetilde{\bm{x}}^{k+1})-F(\bm{x})\big)
+ \vartheta_{k+1}\big(F^*-F(\bm{x})\big),
\end{aligned}
\end{equation*}
where the last inequality follows from $\vartheta_{k+1}:=\theta_{k}^{-2}-\theta_{k+1}^{-2}(1-\theta_{k+1})\geq0$ (by \eqref{condtheta}) and $F(\widetilde{\bm{x}}^{k+1}) \geq F^*$ for all $k\geq0$. Thus, combining the above relations yields \eqref{suffdes-iAPG} and completes the proof.

%%%%%%%%%%%%%%%%%%%%%%%%%%%%%%%%%%%%%%%%%%%%%%
\subsection{Proof of Lemma \ref{lem-smallo}}\label{proof-lem-smallo}

\textit{Statement (i)}.
Let $\bm{x}^*$ be an arbitrary optimal solution of problem \eqref{mainpro}. It then follows from \eqref{suffdes-iAPG-tmp} with $\bm{x}^*$ in place of $\bm{x}$ that, for any $i\geq0$,
\begin{equation*}%\label{suffdes1-iAPG-tmp}
\begin{aligned}
&\quad \big(\theta_{i-1}^{-2}-\theta_i^{-2}+\theta_i^{-1}\big)
\big(F(\widetilde{\bm{x}}^{i})-F(\bm{x}^*)\big) \\[2pt]
&\leq \theta_{i-1}^{-2}\big(F(\widetilde{\bm{x}}^{i})-F(\bm{x}^*)\big)
- \theta_{i}^{-2}\big(F(\widetilde{\bm{x}}^{i+1})-F(\bm{x}^*)\big)
+ {\textstyle\frac{L}{2}}\|\bm{x}^* - \bm{z}^i\|^2
- {\textstyle\frac{L}{2}}\|\bm{x}^* - \bm{z}^{i+1}\|^2 \\[2pt]
&\qquad
+ \big(\theta_i^{-1}\eta_i+(\alpha-1)L\theta_{i-1}^{-1}\mu_{i-1}\big)
\|\bm{x}^*-\bm{z}^{i+1}\|
+ (\alpha-1)L\theta_{i-1}^{-1}\mu_{i-1}\|\bm{x}^*-\bm{z}^i\| + \xi_i.
\end{aligned}
\end{equation*}
Then, for any $k\geq0$, summing the above inequality from $i=0$ to $i=k$ results in
\begin{equation*}%\label{suffdes1-iAPG-tmp}
\begin{aligned}
&\quad {\textstyle\sum^k_{i=0}}\big(\theta_{i-1}^{-2}-\theta_i^{-2}+\theta_i^{-1}\big)
\big(F(\widetilde{\bm{x}}^{i})-F(\bm{x}^*)\big) \\
&\leq \theta_{-1}^{-2}\big(F(\widetilde{\bm{x}}^{0})-F(\bm{x}^*)\big)
- \theta_{k}^{-2}\big(F(\widetilde{\bm{x}}^{k+1})-F(\bm{x}^*)\big)
+ {\textstyle\frac{L}{2}}\|\bm{x}^* - \bm{z}^0\|^2
- {\textstyle\frac{L}{2}}\|\bm{x}^* - \bm{z}^{k+1}\|^2 \\
&~~
+ {\textstyle\sum^{k}_{i=0}}
\big((\theta_i^{-1}\eta_i+(\alpha\!-\!1)L\theta_{i-1}^{-1}\mu_{i-1})\|\bm{x}^*-\bm{z}^{i+1}\|
+ (\alpha\!-\!1)L\theta_{i-1}^{-1}\mu_{i-1}\|\bm{x}^*-\bm{z}^i\| + \xi_i\big) \\
&\leq \theta_{-1}^{-2}\big(F(\widetilde{\bm{x}}^{0})-F(\bm{x}^*)\big)
+ {\textstyle\frac{L}{2}}\|\bm{x}^* - \bm{z}^0\|^2
+ \overline{\delta}_k\|\bm{x}^*-\bm{z}^0\| + B_k,
\end{aligned}
\end{equation*}
where the last inequality follows from \eqref{skbd-iAPG} with $e(\bm{x}^*)=0$, and $B_k:=(2/L)\,\overline{\delta}_k^2+\sqrt{2/L}\,\overline{\delta}_k\,\overline{\xi}_k^{\frac{1}{2}}+\overline{\xi}_k$ with $\overline{\delta}_k$ and $\overline{\xi}_k$ defined in \eqref{defnots}. Moreover, since $\sum\theta_k^{-1}\eta_k<\infty$, $\sum\theta_k^{-1}\mu_k<\infty$, $\sum\theta_k^{-2}\nu_k<\infty$ and $\theta_k\in(0,1]$, one can verify that $\lim\limits_{k\to\infty}\overline{\delta}_k<\infty$ and $\lim\limits_{k\to\infty}B_k<\infty$. Thus, we see that
\begin{equation*} {\textstyle\sum^{\infty}_{k=0}}\big(\theta_{k-1}^{-2}-\theta_k^{-2}+\theta_k^{-1}\big)
\big(F(\widetilde{\bm{x}}^{k})-F(\bm{x}^*)\big) < \infty.
\end{equation*}
This, together with $\theta_{k-1}^{-2}\!-\!\theta_k^{-2}\!+\!\theta_k^{-1}
= \frac{(\alpha-3)k+(\alpha-2)^2}{(\alpha-1)^2}$ and $\alpha>3$, proves statement (i).

\vspace{1mm}
\textit{Statement (ii)}.
Next, setting $\bm{x}=\widetilde{\bm{x}}^k$ in \eqref{suffdes1-iAPG}, we see that, for any $k\geq0$,
\begin{equation*}%\label{suffdesk-iAPG}
F(\widetilde{\bm{x}}^{k+1}) + {\textstyle\frac{L}{2}}\|\bm{x}^{k+1}-\widetilde{\bm{x}}^k\|^2
\leq F(\widetilde{\bm{x}}^k)
+ {\textstyle\frac{L}{2}}\|\bm{y}^{k}-\widetilde{\bm{x}}^k\|^2
+ \eta_k\|\widetilde{\bm{x}}^{k+1}-\widetilde{\bm{x}}^k\|
+ {\textstyle\frac{L}{2}}\mu_k^2 + \nu_k.
\end{equation*}
Note also that
\begin{equation*}
\begin{aligned}
\|\widetilde{\bm{x}}^{k+1}-\widetilde{\bm{x}}^k\|
&=\|\bm{x}^{k+1} \!-\! \bm{x}^k + \widetilde{\bm{x}}^{k+1} \!-\! \bm{x}^{k+1}
+ \bm{x}^k \!-\! \widetilde{\bm{x}}^k\|
\leq \|\bm{x}^{k+1}-\bm{x}^k\| + \mu_k + \mu_{k-1}, \\[2pt]
\|\bm{x}^{k+1}-\widetilde{\bm{x}}^k\|^2
&= \|\bm{x}^{k+1}-\bm{x}^k + \bm{x}^k - \widetilde{\bm{x}}^k\|^2 \\
&= \|\bm{x}^{k+1}-\bm{x}^k\|^2 + 2\langle\bm{x}^{k+1}-\bm{x}^k,\,\bm{x}^k - \widetilde{\bm{x}}^k\rangle + \|\bm{x}^k - \widetilde{\bm{x}}^k\|^2 \\
&\geq \|\bm{x}^{k+1}-\bm{x}^k\|^2 + \|\bm{x}^k - \widetilde{\bm{x}}^k\|^2 - 2\mu_{k-1}\|\bm{x}^{k+1}-\bm{x}^k\|, \\[2pt]
\|\bm{y}^{k}-\widetilde{\bm{x}}^k\|^2
&= \|\bm{x}^k + \theta_k(\theta_{k-1}^{-1}-1)(\bm{x}^k-\bm{x}^{k-1})-\widetilde{\bm{x}}^k\|^2 \\
&= \theta_k^2(\theta_{k-1}^{-1}-1)^2\|\bm{x}^{k}-\bm{x}^{k-1}\|^2 + 2\theta_k(\theta_{k-1}^{-1}-1)\langle\bm{x}^{k}-\bm{x}^{k-1},\bm{x}^k - \widetilde{\bm{x}}^k\rangle  + \|\bm{x}^k - \widetilde{\bm{x}}^k\|^2 \\
&\leq \big({\textstyle\frac{k-1}{k+\alpha-1}}\big)^2\|\bm{x}^{k}\!-\!\bm{x}^{k-1}\|^2 + \|\bm{x}^k \!-\! \widetilde{\bm{x}}^k\|^2 + 2\mu_{k-1}\big({\textstyle\frac{k-1}{k+\alpha-1}}\big)\|\bm{x}^{k}\!-\!\bm{x}^{k-1}\|.
\end{aligned}
\end{equation*}
Then, combining the above inequalities results in
\begin{equation*}
\begin{aligned}
&\quad F(\widetilde{\bm{x}}^{k+1}) + {\textstyle\frac{L}{2}}\|\bm{x}^{k+1}-\bm{x}^k\|^2  \\
&\leq F(\widetilde{\bm{x}}^k)
+ {\textstyle\frac{L}{2}\big(\frac{k-1}{k+\alpha-1}\big)^2}
\|\bm{x}^{k}-\bm{x}^{k-1}\|^2
+ (\eta_k+2\mu_{k-1})\|\bm{x}^{k+1}-\bm{x}^k\| \\
&\quad + 2\mu_{k-1}\big({\textstyle\frac{k-1}{k+\alpha-1}}\big)\|\bm{x}^{k}-\bm{x}^{k-1}\|
+ {\textstyle\frac{L}{2}}\mu_k^2 + \eta_k(\mu_k + \mu_{k-1}) + \nu_k.
\end{aligned}
\end{equation*}
This together with $k+\alpha-1 \geq k+1$ (since $\alpha>3$) implies that, for any $k\geq0$,
\begin{equation*}
\begin{aligned}
&\quad~ (k+1)^2F(\widetilde{\bm{x}}^{k+1}) + {\textstyle\frac{L}{2}}(k+1)^2\|\bm{x}^{k+1}-\bm{x}^k\|^2 \\
&\leq (k+1)^2F(\widetilde{\bm{x}}^k)
+ {\textstyle\frac{L}{2}}(k-1)^2\|\bm{x}^{k}-\bm{x}^{k-1}\|^2
+ (\eta_k + 2\mu_{k-1})(k+1)^2\|\bm{x}^{k+1}-\bm{x}^k\| \\
&\qquad + 2\mu_{k-1}(k-1)(k+1)\|\bm{x}^{k}-\bm{x}^{k-1}\| + (k+1)^2\big({\textstyle\frac{L}{2}}\mu_k^2 + \eta_k(\mu_k + \mu_{k-1}) + \nu_k\big) \\
&\leq (k+1)^2F(\widetilde{\bm{x}}^k)
+ {\textstyle\frac{L}{2}}(k-1)^2\|\bm{x}^{k}-\bm{x}^{k-1}\|^2
+ p_k(k+1)\|\bm{x}^{k+1}-\bm{x}^k\|  \\
&\qquad + \widetilde{p}_{k}(k-1)\|\bm{x}^{k}-\bm{x}^{k-1}\| + s_k,
\end{aligned}
\end{equation*}
where $p_k:=(k+1)(\eta_k + 2\mu_{k-1})$, $\widetilde{p}_{k}:=2(k+1)\mu_{k-1}$ and $s_k:=(k+1)^2\big(\frac{L}{2}\mu_k^2 + \eta_k(\mu_k+\mu_{k-1})+\nu_k\big)$. Subtracting $(k+1)^2F(\bm{x}^*)$ from both sides in the above inequality and rearranging the resulting relation, we obtain that
\begin{equation}\label{recur-k2f}
\begin{aligned}
&\quad (k+1)^2\big(F(\widetilde{\bm{x}}^{k+1})-F(\bm{x}^*)\big) + {\textstyle\frac{L}{2}}(k+1)^2\|\bm{x}^{k+1}-\bm{x}^k\|^2 \\
&\leq k^2\big(F(\widetilde{\bm{x}}^k)-F(\bm{x}^*)\big)
+ {\textstyle\frac{L}{2}}(k\!-\!1)^2\|\bm{x}^{k}-\bm{x}^{k-1}\|^2
+ (2k\!+\!1)\big(F(\widetilde{\bm{x}}^k)-F(\bm{x}^*)\big) \\
&\qquad
+ p_k(k+1)\|\bm{x}^{k+1}-\bm{x}^k\|
+ \widetilde{p}_{k}(k-1)\|\bm{x}^{k}-\bm{x}^{k-1}\|
+ s_k \\
&\leq k^2\big(F(\widetilde{\bm{x}}^k)-F(\bm{x}^*)\big)
+ {\textstyle\frac{L}{2}}k^2\|\bm{x}^{k}-\bm{x}^{k-1}\|^2
+ (2k+1)\big(F(\widetilde{\bm{x}}^k)-F(\bm{x}^*)\big) \\
&\qquad
+ p_k(k+1)\|\bm{x}^{k+1}-\bm{x}^k\|
+ \widetilde{p}_{k}k\|\bm{x}^{k}-\bm{x}^{k-1}\|
+ s_k.
\end{aligned}
\end{equation}
%where the last inequality follows from $k-1<k$.
Then, from this inequality, we see that, for any $k\geq0$,
\begin{equation*}
\begin{aligned}
&\quad (k+1)^2\big(F(\widetilde{\bm{x}}^{k+1})-F(\bm{x}^*)\big) + {\textstyle\frac{L}{2}}(k+1)^2\|\bm{x}^{k+1}-\bm{x}^k\|^2 \\
%&\leq {\textstyle\sum^k_{i=0}}\left((2i+1)\big(F(\widetilde{\bm{x}}^i)-F(\bm{x}^*)\big)
%+ p_i(i+1)\|\bm{x}^{i+1}-\bm{x}^i\|
%+ \widetilde{p}_{i}i\|\bm{x}^{i}-\bm{x}^{i-1}\|
%+ s_i \right) \\
&\leq {\textstyle\sum^k_{i=0}}(2i+1)\big(F(\widetilde{\bm{x}}^i)-F(\bm{x}^*)\big)+ {\textstyle\sum^k_{i=0}}\big(p_i(i+1)\|\bm{x}^{i+1}-\bm{x}^i\|
+ \widetilde{p}_{i}i\|\bm{x}^{i}-\bm{x}^{i-1}\|
+ s_i\big),
\end{aligned}
\end{equation*}
which, together with $F(\widetilde{\bm{x}}^{k+1})-F(\bm{x}^*)\geq0$, implies that, for any $k\geq0$,
\begin{equation*}
\begin{aligned}
{\textstyle\frac{L}{2}}(k+1)^2\|\bm{x}^{k+1}-\bm{x}^k\|^2
&\leq {\textstyle\sum^k_{i=0}}(2i+1)\big(F(\widetilde{\bm{x}}^i)-F(\bm{x}^*)\big) \\
&\quad + {\textstyle\sum^k_{i=0}}\big(p_i(i+1)\|\bm{x}^{i+1}-\bm{x}^i\|
+ \widetilde{p}_{i}i\|\bm{x}^{i}-\bm{x}^{i-1}\| + s_i\big).
\end{aligned}
\end{equation*}
Thus, applying Lemma \ref{lem-recursion} with $a_k:=\sqrt{L/2}\,k\|\bm{x}^{k}-\bm{x}^{k-1}\|$, $q_k\!:=\!\sum^k_{i=0}(2i+1)\big(F(\widetilde{\bm{x}}^i)-F(\bm{x}^*)\big)$, $c_k:=s_k$, $\lambda_k:=\sqrt{2/L}\,p_k$ and $\widetilde{\lambda}_k:=\sqrt{2/L}\,\widetilde{p}_k$, and noticing that $\{q_k\}$ is nondecreasing in this case, we get
\begin{equation}\label{skbd1-iAPG-tmp1}
\begin{aligned}
&\quad {\textstyle\sum^k_{i=0}}\big(p_i(i+1)\|\bm{x}^{i+1}-\bm{x}^i\|
+ \widetilde{p}_{i}i\|\bm{x}^{i}-\bm{x}^{i-1}\| + s_i\big) \\
&\leq {\textstyle\frac{2}{L}}\!\left({\textstyle\sum^k_{i=0}}\big(p_i
\!+\!\widetilde{p}_i\big)\right)^2
\!+\! \sqrt{{\textstyle\frac{2}{L}}}\,{\textstyle\sum^k_{i=1}}\!\!\left(\!\big(p_i
\!+\!\widetilde{p}_i\big)\sqrt{{\textstyle\sum^i_{j=0}}(2j\!+\!1)
\big(F(\widetilde{\bm{x}}^j)\!-\!F(\bm{x}^*)\big)}\right) \\
&\qquad
+\sqrt{{\textstyle\frac{2}{L}}}\left({\textstyle\sum^k_{i=0}}\big(p_i
+\widetilde{p}_i\big)\right)\left({\textstyle\sum^k_{i=0}}s_i\right)^{\frac{1}{2}}
+ {\textstyle\sum^k_{i=0}}s_i.
\end{aligned}
\end{equation}
Moreover, since $\sum\theta_k^{-1}\eta_k<\infty$ and $\sum\theta_k^{-1}\mu_k<\infty$, one can verify that $\sum p_i<\infty$, $\sum \widetilde{p}_i<\infty$ and $\sum s_i < \infty$. This together with statement (i) and \eqref{skbd1-iAPG-tmp1} implies that
\begin{equation}\label{sumfseqbd}
\sum\Big((2i+1)\big(F(\widetilde{\bm{x}}^i)-F(\bm{x}^*)\big)
+ p_i(i+1)\|\bm{x}^{i+1}-\bm{x}^i\|
+ \widetilde{p}_{i}i\|\bm{x}^{i}-\bm{x}^{i-1}\| + s_i \Big)
< \infty.
\end{equation}
Now, using the first inequality in \eqref{recur-k2f} and $k^2+k+1\leq(k+1)^2$, we have that
\begin{equation*}%\label{recur2-k2f}
\begin{aligned}
&\quad (k+1)^2\big(F(\widetilde{\bm{x}}^{k+1})-F(\bm{x}^*)\big) + {\textstyle\frac{L}{2}}k^2\|\bm{x}^{k+1}-\bm{x}^k\|^2 + {\textstyle\frac{L}{2}}(k+1)\,\|\bm{x}^{k+1}-\bm{x}^k\|^2 \\
&\leq (k+1)^2\big(F(\widetilde{\bm{x}}^{k+1})-F(\bm{x}^*)\big) + {\textstyle\frac{L}{2}}(k+1)^2\|\bm{x}^{k+1}-\bm{x}^k\|^2 \\
&\leq k^2\big(F(\widetilde{\bm{x}}^k)-F(\bm{x}^*)\big)
+ {\textstyle\frac{L}{2}}(k-1)^2\|\bm{x}^{k}-\bm{x}^{k-1}\|^2
+ (2k+1)\big(F(\widetilde{\bm{x}}^k)-F(\bm{x}^*)\big) \\
&\qquad
+ p_k(k+1)\|\bm{x}^{k+1}-\bm{x}^k\|
+ \widetilde{p}_{k}k\|\bm{x}^{k}-\bm{x}^{k-1}\| + s_k
\end{aligned}
\end{equation*}
and thus
\begin{equation*}%\label{recur2-k2f}
\begin{aligned}
&\quad (k+1)^2\big(F(\widetilde{\bm{x}}^{k+1})-F(\bm{x}^*)\big) + {\textstyle\frac{L}{2}}k^2\|\bm{x}^{k+1}-\bm{x}^k\|^2
+ {\textstyle\frac{L}{2}}{\textstyle\sum^k_{i=0}}(i+1)\|\bm{x}^{i+1}-\bm{x}^i\|^2\\
&\leq {\textstyle\sum^k_{i=0}}\left((2i+1)\big(F(\widetilde{\bm{x}}^i)-F(\bm{x}^*)\big)
+ p_i(i+1)\|\bm{x}^{i+1}-\bm{x}^i\|
+ \widetilde{p}_{i}i\|\bm{x}^{i}-\bm{x}^{i-1}\| + s_i \right).
\end{aligned}
\end{equation*}
Combining this and \eqref{sumfseqbd} yields $\sum k\|\bm{x}^{k}-\bm{x}^{k-1}\|^2 < \infty$ and proves statement (ii).

\vspace{1mm}
\textit{Statement (iii)}.
It is readily follows from \eqref{recur-k2f}, \eqref{sumfseqbd} and Lemma \ref{lemseqcon} that $\big\{k^2\big(F(\widetilde{\bm{x}}^k)-F^*\big) + \frac{L}{2}k^2\|\bm{x}^{k}-\bm{x}^{k-1}\|^2\big\}$ is convergent and hence $\lim\limits_{k\to\infty}k^2\big(F(\widetilde{\bm{x}}^k)-F^*\big) + \frac{L}{2}k^2\|\bm{x}^{k}-\bm{x}^{k-1}\|^2$ exists. We then complete the proof.

\end{appendices}

%%%%%%%%%%%%%%%%%%%%%%%%%%%%%%%%%%%%%%%%%%
\section*{Acknowledgments}

\noindent The research of Lei Yang is supported by the National Natural Science Foundation of China under grant 12301411. The research of Meixia Lin is supported by the Ministry of Education, Singapore, under its Academic Research Fund Tier 2 grant call (MOE-T2EP20123-0013).

%%%%%%%%%%%%%%%%%%%%%%%%%%%%%%%%%%%%%%%%%%
%% References
\bibliographystyle{plain}
\bibliography{references/Ref_SpinAPG}

\end{document}